\documentclass[reqno, colorBG]{amsart}
\usepackage{amssymb}
\usepackage{amsmath}
\usepackage[mathscr]{euscript}

\usepackage[small]{caption}
\usepackage{psfrag}
\usepackage{graphicx}
\graphicspath{ {images/} }
\usepackage{color}

\makeatletter
\@addtoreset{equation}{section}
\makeatother

\newtheorem{theorem}{Theorem}[section]
\newtheorem{lemma}[theorem]{Lemma}
\newtheorem{proposition}[theorem]{Proposition}

\newtheorem{remark}[theorem]{Remark}

\newcommand{\mc}[1]{{\mathcal #1}}
\newcommand{\mf}[1]{{\mathfrak #1}}

\newcommand{\bb}[1]{{\mathbb #1}}
\newcommand{\bs}[1]{{\boldsymbol #1}}
\newcommand{\ms}[1]{{\mathscr #1}}

\begin{document}

\title[Metastability of mean-field Potts Model]
{Metastability of Non-reversible, Mean-field Potts Model with Three Spins}

\begin{abstract}
  We examine a non-reversible, mean-field Potts model with three spins
  on a set with $N\uparrow\infty$ points. Without an external field,
  there are three critical temperatures and five different metastable
  regimes. The analysis can be extended by a perturbative argument to
  the case of small external fields. We illustrate the case of large
  external fields with some phenomena which are not present in the
  absence of external field.
\end{abstract}

\author{C. Landim, I. Seo}

\address{\noindent IMPA, Estrada Dona Castorina 110, CEP 22460 Rio de
  Janeiro, Brasil and CNRS UMR 6085, Universit\'e de Rouen, Avenue de
  l'Universit\'e, BP.12, Technop\^ole du Madril\-let, F76801
  Saint-\'Etienne-du-Rouvray, France.  \newline e-mail: \rm
  \texttt{landim@impa.br} }

\address{\noindent Courant Institute of Mathematical Sciences, New York
  University, 251 Mercer Street, New York, N.Y. 10012-1185, USA.
  \newline e-mail: \rm \texttt{insuk@cims.nyu.edu} }


\maketitle

\section{Introduction\label{sec1}}

Some recent progress has been achieved in the potential theory of
non-reversible Markov chains. Gaudilli\`ere and Landim \cite{GL}
obtained a variational formula for the capacity between two disjoint
sets, expressed as a double infima over flows satisfying certain
boundary conditions, and Slowik \cite{Slo} showed that the capacity
can also be represented as a double suprema over flows satisfying a
different set of boundary conditions.

These advancements permitted to describe the metastable behavior of
some non-reversible dynamics. The evolution of the condensate in a
totally asymmetric zero range process on a finite torus has been
examined in \cite{Lan2}, and the behavior of the ABC model among the
segregated configurations in the zero temperature limit has been
derived by Misturini in \cite{Mis}, applying the martingale method
introduced in \cite{BL1, BL2}.

In a previous paper \cite{LS}, inspired by the mean-field Potts model
presented in this article and based on the variational formulae
alluded to above, we characterized the metastable behavior of
non-reversible, continuous-time random walks in a potential field,
extending to the irreversible setting results obtained by Bovier,
Eckhoff, Gayrard, Klein for reversible diffusions \cite{BEGK1, BEGK2}
and by Landim, Misturini, Tsunoda for reversible random walks in a
potential field \cite{LMT}. Among other results, we proved the
Eyring-Kramers formula \cite{Ber} for the transition rate between a
metastable set and a stable set, predicted by Bouchet, Reygner
\cite{BR} in the context of irreversible diffusion processes.

We examine in this article a non-reversible, mean-field Potts model
\cite{Pot, Wu} with three spins. In the same way as the mean-field
Ising model is mapped to a nearest-neighbor, one-dimensional random
walk on a potentiel field \cite{CGOV}, the dynamics of the mean-field
Potts model can be mapped to a non-reversible random walk on a
two-dimensional simplex.

If there is no external magnetic field, three critical temperatures and five
different metastable regimes are observed. We refer to Figure
\ref{fig6} for an illustration of the potential in each regime. There
exists a temperature $0<T_3<\infty$ above which no metastable behavior
is observed because in this regime the entropy prevails over the
energy. If the temperature $T$ is greater than or equal to $T_3$, in a
typical configuration, one third of the spins takes one of the
possible values of the spin, and starting from any configuration the
system is driven progressively to this state.

There is a second critical temperature, denoted by $T_2$, at which
four metastable sets coexist. The first one corresponds to the
configurations in which one third of the spins takes one of the
possible values of the spin, while the other three correspond to the
configurations in which a large majority of the spins takes one of the
spins value. We call the first metastable set the entropic one, and
the last three metastable sets the energetic ones.  The dynamics among
the metastable sets can be described by a $4$-state Markov chain whose
graph has a star shape. In this reduced model, jumps from a point
which represents an energetic metastable set to a similar point are
not allowed. Hence, to go from an energetic point to another, the
reduced chain must visit the entropic point.

In the temperature range $(T_2,T_3)$, there are three metastable sets
which correspond, in the terminology introduced in the previous
paragraph, to the energetic sets, and one stable set, the entropic
set. In this regime, in an appropriate time scale, starting from an
energetic set, after an exponential time, the process jumps to the
stable set and their remains for ever. Therefore, this evolution can
be represented by a $4$-state Markov chain whose graph has a star
shape and whose center is an absorbing point.

There is a third critical temperature, denoted by $T_1$. At this
temperature there are three metastable sets, the so-called energetic
ones. These three metastable sets are separated by a unique critical
point, but the Hessian of the potential at this critical point is the
zero matrix. In particular, this point is not a saddle point and the
approach developed in \cite{LS} is not useful to prove the metastable
behavior of this dynamics. Hence, even if we believe that a metastable
behavior occurs among the three energetic sets, the existing
techniques do not cover this situation.

In the temperature range $(T_1,T_2)$, there are four metastable sets
and two time scales. The entropic set is shallower than the energetic
ones and in a certain time scale, starting from the entropic set,
after an exponential time the process jumps with equal probability to
one of the energetic sets and there remains for ever. In a longer time
scale, the metastable behavior of this dynamics can be described by a
$3$-state Markov chain whose graph is the complete graph. 

Actually, in this range of temperatures a remarkable phenomenon
occurs.  Starting from one of the energetic sets, after an exponential
time the chain jumps to the entropic set. Once at the entropic set,
the chain immediately jumps to one of the energetic sets with equal
probability, and repeat from there the evolution just
described. Hence, to move from one energetic set to another, the
dynamics first dismantles the spin alignment present in the energetic
set, staying during a negligible amount of time in the entropic set,
and then, almost instantaneously, rebuild a new alignment which can
coincide with the one existing before the visit to the entropic set.

Finally, in the temperature range $(0,T_1)$, the entropic set
disappears and only the three energetic sets remain. As in the
temperature range $(T_1,T_2)$, the evolution among these sets can be
described by a $3$-state Markov chain whose graph is the complete
graph. The difference with the previous case is that the three saddle
points of the potential separate here the energetic sets, while in the
previous case these saddle points separate the energetic sets from the
entropic set.

A perturbative argument permits to extend the previous analysis to the
case in which the external field is small. In this case, of course,
the external field breaks the symmetry among the energetic sets, and
one or two of them may be favored. Besides this fact, the qualitative
behavior of the dynamics is similar to the one without external field.

The analysis of the metastable behavior of a random walk in a
potential field proposed in \cite{BEGK1, BEGK2, LMT, LS} relies on the
identification of the critical points of the potential and on the
characterization of the eigenvalues of the Hessian of the potential at
the critical points. It is not possible, in general, to obtain
explicit expressions for the critical points of the potential induced
by the non-reversible, mean-field Potts model. For this reason a global
rigorous investigation of the metastable behavior with non small
external magnetic field is not possible. However, in the case where
the direction of the magnetic field points in the direction or opposite direction
of one of the three possible values of the spins, a complete
description of the metastable behavior of the Potts model is possible.
This is presented in the last section of the article, as well as some
phenomenon not observed at zero external field which are supported by
numerical computations.

\section{Model and Results}
\label{sec2}

\subsection{Mean-field Potts Model }

Let $\mc{S} = \{\bs{v}_{0},\, \bs{v}_{1},\, \bs{v}_{2}\}$ be the set
of spins, where $\bs{v}_{k}=\left(\cos(2\pi k/3),\,\sin(2\pi
  k/3)\right)$, $0\le k\le2$, and let $T_{N}=\{1,\,2,\,\dots,\,N\}$, $N\in \mathbb{N}$ be 
the set of sites. The configuration space, represented by
$\Omega_{N}$, is the set $\mc{S}^{T_{N}}$.  Denote by
$\sigma=(\sigma_{1},\sigma_{2},\dots,\sigma_{N})$ the configurations
of $\Omega_{N}$, where $\sigma_{i}\in\mc{S}$, $i\in T_{N}$, is the
spin at the $i$-th site of $\sigma$. The Hamiltonian
$\mathbb{H}_{N}:\Omega_{N}\rightarrow\mathbb{R}$ is defined by
\begin{equation}
\label{ham1}
\mathbb{H}_{N}(\sigma)\;=\; -\, \frac{1}{2N}\sum_{1\le i,j\le N}
\sigma_{i}\cdot\sigma_{j} \;-\; \sum_{i=1}^{N} \bs{h}_{\text{e}}\cdot\sigma_{i}
\;=\;-\frac{N}{2} \, \Big|\frac{1}{N}\sum_{i=1}^{N}
\sigma_{i}\Big|^{2} \;-\; \bs{h}_{\text{e}}\cdot \sum_{i=1}^{N}\sigma_{i}\;,
\end{equation}
where $\bs{h}_{\text{e}} = (r_{\text{e}}\cos\theta_{\text{e}},
\,r_{\text{e}}\sin\theta_{\text{e}})$ stands for an external magnetic
field, and $\bs x \cdot \bs y$ for the scalar product between $\bs x$
and $\bs y\in\bb R^2$. Here, $r_{\text{e}} \ge0$ and
$0\le\theta_{\text{e}}<2\pi$ represent the magnitude and the angle of
the external field, respectively. The model associated to this
mean-field type Hamiltonian is known as the mean-field Potts Model
\cite{Pot}. We refer to the review paper \cite{Wu} for an introduction
on Potts model.

Fix $\beta>0$ and denote by $\mu_{\beta}^{N}$ the Gibbs measure
associated to the Hamiltonian $\mathbb{H}_{N}$ at the inverse
temperature $\beta$:
\begin{equation}
\label{gibbs}
\mu_{\beta}^{N}(\sigma)\;=\;\frac{3^{-N}}{Z_{N}(\beta)}
e^{-\beta\mathbb{H}_{N}(\sigma)}\;\;;\;\sigma\in\Omega_N\;,
\end{equation}
where $Z_{N}(\beta)$ is the partition function defined by 
\begin{equation*}
Z_{N}(\beta)\;=\;3^{-N}\sum_{\sigma\in\Omega_{N}}e^{-\beta\mathbb{H}_{N}(\sigma)}
\end{equation*}
so that $\mu_{\beta}^{N}$ is a probability measure on $\Omega_{N}$.

\subsection{Spin Dynamics}

A natural dynamics for the Potts model introduced in the previous
section is the one in which spins are allowed to jump only in one
direction, say the counter-clockwise one: $\bs{v}_{k}
\rightarrow\bs{v}_{k+1}$, $0\le k\le2$, where summation in the
subscript is performed modulo $3$.  Denote by $\mathscr{R}:\mc
S\rightarrow \mc S$ the counter-clockwise rotation on $\mc S$, i.e.,
$\mathscr{R}(\bs{v}_k)=\bs{v}_{k+1}$ for $0\le k\le 2$, and denote by
$\tau_{i}\,\sigma$, $i\in T_{N}$, the configuration obtained from
$\sigma$ by rotating counter-clockwise the $i$-th spin by an angle of
$(2\pi/3)$, namely,
\begin{equation*}
(\tau_{i}\,\sigma)_{j}\;=\;\mathscr{R}(\sigma_{j})\,\mathbf{1}\{j=i\}+\sigma_{j}\,\mathbf{1}\{j\neq i\}\;.
\end{equation*}

Denote by $\ms L_N$ the generator which acts on functions
$f:\Omega_{N}\rightarrow\mathbb{R}$ as
\begin{equation}
\label{gen1}
(\mathcal{L}_{N}f)(\sigma)\;=\;\frac{1}{N}\sum_{i=1}^{N}
c_{i}(\sigma)\, \{f(\tau_{i}\sigma)-f(\sigma)\}\;,
\end{equation}
where $c_{i}(\sigma)$, $i\in T_{N}$, is the jump rate given by
\begin{equation}
\label{rate1}
c_{i}(\sigma)\;=\;\exp\Big\{ -\frac{\beta}{3}\sum_{k=0}^{2}
\left[\mathbb{H}_{N}(\tau_{i}^{(k)}\sigma)-
\mathbb{H}_{N}(\sigma)\right]\Big\} \;,
\end{equation}
and where $\tau_{i}^{(k)}$, $k\ge 0$, stands for the $k$-th iterated
of the operator $\tau_{i}$.  These jump rates were chosen for
$\mu_{\beta}^{N}$ to be the stationary state.

Denote by $\sigma(t)=(\sigma_{1}(t),\cdots,\sigma_{N}(t))$, $t\ge0$,
the continuous-time Markov chain on $\Omega_{N}$ generated by
$\mathscr{L}_{N}$.  Note that $\sigma(t)$ is non-reversible with
respect to $\mu_{N}^{\beta}$ because of the cyclic nature of the
dynamics.

\subsection{Metastability}

Denote by $\mf{m}_{N}(\sigma)$ the magnetization of the configuration
$\sigma\in\Omega_{N}$: 
\begin{equation*}
\mf{m}_{N}(\sigma)\;=\;\frac{1}{N}\sum_{i=1}^{N}\sigma_{i}\;.
\end{equation*}
In this article, we investigate the metastable behavior of the the
magnetization $\mf{m}_{N}(t) := \mf{m}_{N}(\sigma(t))$ under the
dynamics defined by (\ref{gen1}).

Note that the Hamiltonian (\ref{ham1}) can be represented in terms
of the magnetization:
\begin{equation}
\label{ham2}
\mathbb{H}_{N}(\sigma)\;=\; \, N \Big[- \frac{1}{2}
\left|\mf{m}_{N}(\sigma)\right|^{2} \;-\;
\bs{h}_{\text{e}}\cdot\mf{m}_{N}(\sigma)\Big]\;,
\end{equation}
and that the rotation rate $c_{i}(\sigma)$ is represented only in
terms of the Hamiltonian. Thereby, the process $\mf{m}_{N}(t)$ is
itself a continuous-time Markov chain on $\mathbb{R}^{2}$ and inherits
the non-reversibility from the underlying spin dynamics.

It has been observed in \cite{CGOV, BEGK2} that the magnetization of
the mean-field, Curie-Weiss model exhibits a metastable behavior at
low temperatures due to the competition between entropy and
energy. The mean-field, Potts model considered in this article can be
regarded as a generalization of the Curie-Weiss model, and exhibits an
analogous metastable behavior at low temperatures. 

A complete analysis of the metastable behavior in the case where there
is no external field is presented in Section \ref{sec4}. As mentioned
in the introduction, there exist in this case three critical inverse
temperatures $\beta_{1}>\beta_{2}>\beta_{3}$, where $\beta_i$ stands
for the inverse of the temperature $T_i$ referred to in Section
\ref{sec1}. While $\beta_{1}=2$, numerical computations give that
$\beta_{2}\thickapprox 1.8484$ and $\beta_{3}\thickapprox 1.8304$. At
each of these critical temperatures a qualitative modification of the
metastable behavior is observed.

The article is organized as follows. In Section \ref{sec3}, we show
that the evolution of the magnetization is described by a random walk
evolving in a potential field defined in a two-dimensional simplex. In
Section \ref{sec4}, we describe all different metastable regimes in
the case of zero-external field, following the martingale approach of
\cite{BL1, BL2, BL3} and based on the recent work \cite{LS}. In Section
\ref{sec5}, by a perturbative argument, we extend these results to the
case of a small external field, and we present new phenomena which
occur when there is a large external field.

\section{Reduction to a Cyclic Random Walk in a Potential Field}
\label{sec3}

We examine in this section the dynamics of the magnetization $\mf
m_N(t)$. We show that it evolves according to a non-reversible random
walk in a potential field. 

Denote by $r_{N}^{k}(\sigma)$, $0\le k\le2$, the ratio of sites
of $\sigma\in\Omega_N$ whose spin is equal to $\bs{v}_{k}$: 
\begin{equation*}
r_{N}^{k}(\sigma)\;=\;\frac{1}{N}\sum_{i=1}^{N}
\mathbf{1}\{\sigma_{i}=\bs{v}_{k}\}\;.
\end{equation*}
Clearly, for all configurations $\sigma$, $\sum_{0\le k\le 2}
r_{N}^{k}(\sigma) =1$. For this reason, denote by $\Xi$ the
two-dimensional simplex given by
\begin{equation*}
\Xi\;=\;\{\bs{x}=(x_{1},x_{2}):x_{1},\,x_{2}\ge0,\,x_{1}+
x_{2}\le1\}\subset\mathbb{R}^{2}\;,
\end{equation*}
and let $\Xi_{N}$ be the discretization of $\Xi$:
$\Xi_N=\Xi\cap(\mathbb{Z}/N)^2$. A point $(x_{1},x_{2})$ in $\Xi$ or
$\Xi_N$ is represented as $\bs{x}=(x_{1},x_{2})$, and for a point
$(x_{1},x_{2})$, $x_{0}$ stands for $1-x_{1}-x_{2}$.

Let $\bs{r}_{N}(\sigma)=(r_{N}^{1}(\sigma),\,r_{N}^{2}(\sigma)) \in
\Xi_N$. An elementary computation shows that the magnetization can be
expressed in terms of $\bs{r}_{N}(\sigma)$ as
\begin{equation}
\label{mr}
\mf{m}_{N}(\sigma)\;=\;\psi(\bs{r}_{N}(\sigma))\;,
\end{equation}
where $\psi:\Xi\rightarrow\mathbb{R}^{2}$ is defined by
\begin{equation}
\label{psi}
\psi(\bs{x})\;=\;(2x_{1}+x_{2}-1)\bs{v}_{1}+(x_{1}+2x_{2}-1)\bs{v}_{2}\;,
\end{equation}
which is a bijection between $\Xi_{N}$ and $\psi(\Xi_{N})$. Figure
\ref{fig1} illustrates this bijective relation.

\begin{figure}
  \protect
\includegraphics[scale=0.05]{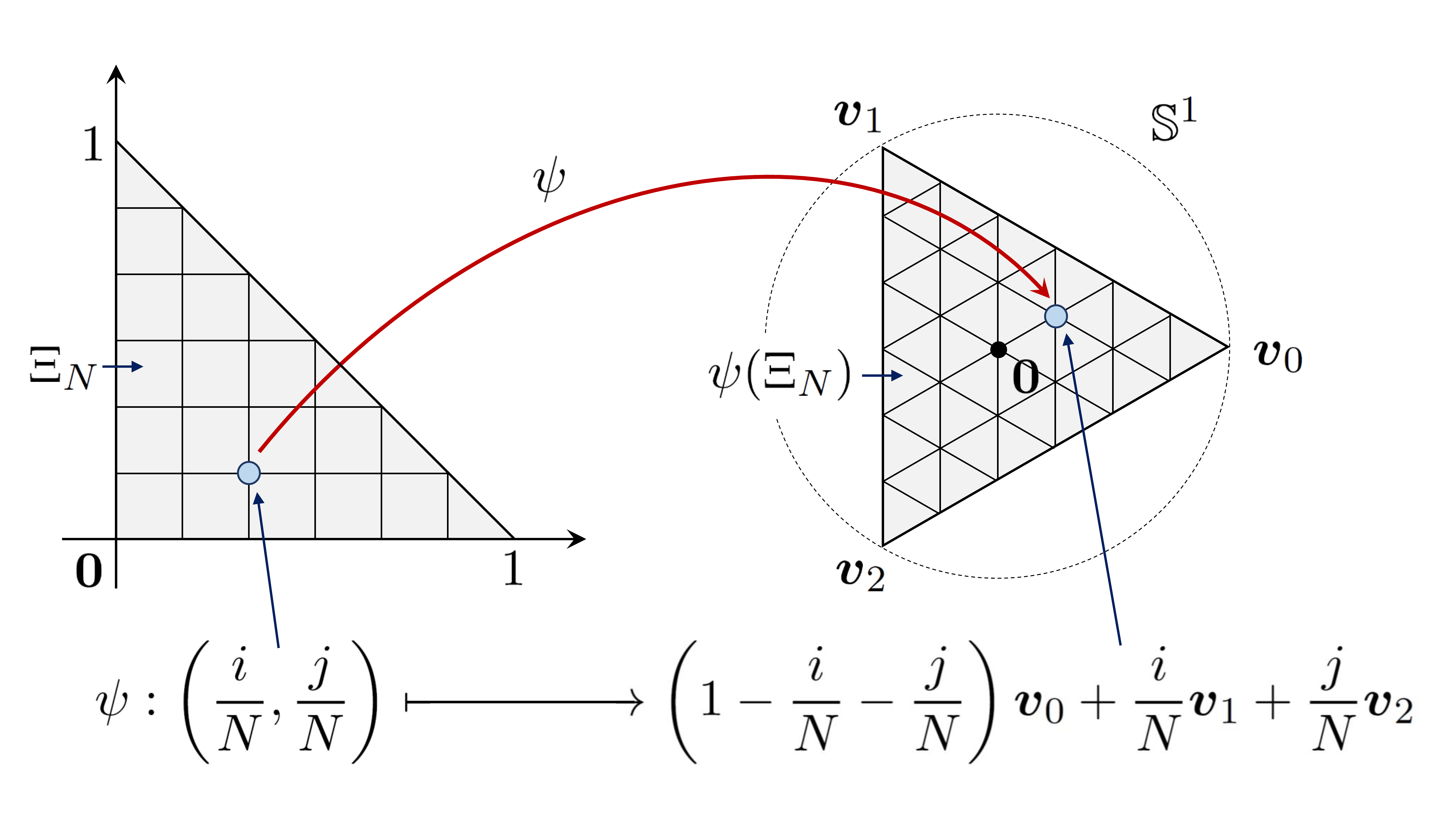}\protect
  \caption{\label{fig1}The bijective relation between
    $\Xi_{N}$ and $\psi(\Xi_{N})$.  The set $\psi(\Xi_{N})$ consists
    of the triangular lattice points of the equilateral triangle
    determined by the three vertices $\bs{v}_{0}$, $\bs{v}_{1}$ and
    $\bs{v}_{2}$.}
\end{figure}

Since $\psi$ is a bijection, to investigate the metastable behavior of
the magnetization $\mf{m}_{N}(t)$, it suffices to examine the
evolution of $\bs{r}_{N}(t):=\bs{r}_{N}(\sigma(t))$. 

\noindent\textbf{The dynamics of $\bs{r}_{N}(t)$}. 
As $\psi$ is a bijection, $\bs{r}_{N}(t)$ inherits the Markov property
from $\mf{m}_{N}(t)$.  We first consider the stationary state of the
dynamics.

By (\ref{ham2}) and (\ref{mr}), the Hamiltonian $\bb H_{N}(\sigma)$
can be written as
\begin{equation}
\label{ham3}
\mathbb{H}_{N}(\sigma)\;=\;NH(\bs{r}_{N}(\sigma))\;,\;\;\;
H(\bs{x})\;=\;-\frac{1}{2}\left|\psi(\bs{x})
\right|^{2}-\bs{h}_{\text{e}}\cdot\psi(\bs{x})\;.
\end{equation}
Hence, the invariant measure of the chain $\bs{r}_{N}(t)$, denoted by
$\nu_{\beta}^{N}$, can be derived from (\ref{gibbs}) and
(\ref{ham3}). More precisely, for $\bs{x}\in\Xi_{N}$, 
\begin{equation}
\label{inv1}
\nu_{\beta}^{N}(\bs{x}) \;=\;
\sum_{\sigma:\bs{r}_{N}(\sigma)=\bs{x}}
\frac{3^{-N}}{Z_{N}(\beta)}\,\exp\{-\beta\mathbb{H}_{N}(\sigma)\}\;.
\end{equation}
Therefore, by straightforward computations and Stirling's formula, 
\begin{equation}
\label{inv2}
\nu_{\beta}^{N}(\bs{x})\;=\;\frac{1}{\widehat{Z}_{N}(\beta)}\,
\exp\{-\beta NF_{\beta,N}(\bs{x})\}\;,
\end{equation}
where $\widehat{Z}_{N}(\beta)=2\pi NZ_{N}(\beta)$ is the partition
function, and where the potential $F_{\beta,N}(\cdot)$ is given by
\begin{equation}
\label{pot1}
F_{\beta,N}(\bs{x})\;=\;F_{\beta}(\bs{x}) \;+\; \frac{1}{N}\,
G_{\beta,N}(\bs{x})\;.
\end{equation}
In this equation,
\begin{equation}
\label{pot2}
F_{\beta}(\bs{x})\;=\;H(\bs{x}) \;+\; \frac{1}{\beta}S(\bs{x})
\;, \quad G_{\beta,N}(\bs{x})\;=\;
\frac{\log(x_{0}x_{1}x_{2})}{2\beta} \;+\; O(N^{-1})\;,
\end{equation}
and $S$ is the entropy function defined by
\begin{equation}
\label{ent}
S(\bs{x})\;=\;\sum_{i=0}^{2}x_{i}\thinspace\log(3x_{i})\;.
\end{equation}
In these equations we used the convention that $\log0=-\infty$ and
that $e^{-\infty} =0$. Moreover, $G_{\beta,N} \rightarrow G_{\beta}$
uniformly on every compact subsets of $\mbox{int}(\Xi)$, where
$G_{\beta}(\bs{x})= \log(x_{0}x_{1}x_{2}) /(2\beta)$.  \smallskip

To examine the dynamics of the chain $\bs{r}_{N}(t)$, denote by
$\bs{e}_{1}=(1,0)$, $\bs{e_{2}}=(0,1)$ the canonical basis of
$\mathbb{R}^{2}$, set $\bs{e}_{0}=(0,0)$, and let
$\bs{e}_{k}^{N}=N^{-1}\bs{e}_{k}$, $0\le k\le2$.  Recall that we
denote by $\bs{v}_{0}$, $\bs{v}_{1}$, $\bs{v}_{2}$ the three values a
spin may assume and that the dynamics allows only jumps from $\bs v_k$
to $\bs v_{k+1}$, $0\le k\le 2$. A jump of spin from $\bs v_0$ to $\bs v_{1}$ corresponds to that of the chain $\bs{r}_{N}(t)$ from $\bs x$ to
$\bs{x}+\bs{e}_{1}^{N}$. Since there are $Nx_{0}$ sites whose spin is
$\bs{v}_{0}$, in view of \eqref{rate1} and \eqref{ham3}, the rate
at which the chain $\bs{r}_{N}(t)$ jumps from $\bs x$ to
$\bs{x}+\bs{e}_{1}^{N}$, denoted by $R_{N}(\bs{x},\bs{x} +
\bs{e}_{1}^{N})$, is given by
\begin{equation*}
R_{N}(\bs{x},\bs{x}+\bs{e}_{1}^{N})\;=\;
x_{0}\exp\left\{ -N\beta(\overline{H}(\bs{x})-H(\bs{x}))\right\} \;,
\end{equation*}
where 
\begin{equation*}
\overline{H}(\bs{x}) \;=\; \frac 13 \big\{ H(\bs{x}) + 
H(\bs{x} + \bs{e}_{1}^{N}) + H(\bs{x}+\bs{e}_{2}^{N}) \big\}\;.
\end{equation*}

Similarly, a jump from $\bs v_1$ to $\bs v_2$ corresponds to a jump of
the chain $\bs{r}_{N}$ from $\bs x$ to
$\bs{x}-\bs{e}_{1}^{N}+\bs{e}_{2}^{N}$, while a jump from $\bs v_2$ to
$\bs v_0$ corresponds to a jump from $\bs x$ to
$\bs{x}-\bs{e}_{2}^{N}$. The rates can be computed easily and are
given by
\begin{align*}
& R_{N}(\bs{x},\bs{x}-\bs{e}_{1}^{N}+\bs{e}_{2}^{N}) \;=\; 
(x_{1}+N^{-1})\, \exp\left\{
  -N\beta(\overline{H}(\bs{x}-\bs{e}_{1}^{N})
-H(\bs{x}))\right\} \;,\\
& \quad R_{N}(\bs{x},\bs{x}-\bs{e}_{2}^{N}) \;=\;  
(x_{2}+N^{-1})\, \exp\left\{ -N\beta(\overline{H}
(\bs{x}-\bs{e}_{2}^{N})-H(\bs{x}))\right\} \;.
\end{align*}
Hence, the generator $\mc{L}_{N}$ of the Markov chain $\bs{r}_{N}(t)$
is given by
\begin{equation}
\label{gen3}
\begin{aligned}
(\mc{L}_{N}f) (\bs{x}) \; & = \;R_{N}(\bs{x},\bs{x}+\bs{e}_{1}^{N})
\left[f(\bs{x}+\bs{e}_{1}^{N})-f(\bs{x})\right] \\
&  +\; R_{N}(\bs{x},\bs{x}-\bs{e}_{1}^{N}+\bs{e}_{2}^{N})
\left[f(\bs{x}-\bs{e}_{1}^{N}+\bs{e}_{2}^{N})-f(\bs{x})\right] \\
&  +\; R_{N}(\bs{x},\bs{x}-\bs{e}_{2}^{N})
\left[f(\bs{x}-\bs{e}_{2}^{N})-f(\bs{x})\right]\;.
\end{aligned}
\end{equation}

Denote by $\mathbb{P}_{\bs{x}}^{N}$ the law of the Markov chain
$\bs{r}_{N}(t)$ starting from $\bs{x}\in\Xi_{N}$ and by
$\mathbb{E}_{\bs{x}}^{N}$ the associated expectation.

\smallskip
\noindent\textbf{Cyclic random walks in a potential
field}. 
Let $\gamma^{N}$ be the cycle $(\bs{e}_{0}^{N}, \bs{e}_{1}^{N},
\bs{e}_{2}^{N}, \bs{e}_{0}^{N})$ on $(\mathbb{Z}/N)^{2}$, and denote by
$\gamma_{\bs{x}}^{N}$ the cycle $\gamma^{N}$ translated by $\bs x\in
(\mathbb{Z}/N)^{2}$, i.e., $\gamma_{\bs{x}}^{N} =\bs{x} +\gamma^{N}$. Let
$\widehat{\Xi}_{N}$ be the set defined by
\begin{equation*}
\widehat{\Xi}_{N} \;=\;
\{\bs{x}\in\Xi_{N}:\gamma_{\bs{x}}^{N}\subset\Xi_{N}\}
\;=\; \{\bs{x}\in\Xi_{N}:x_{1}+x_{2}\le1-N^{-1}\}\;.
\end{equation*}

Denote by $\mc{L}_{N,\bs{x}}$, $\bs{x}\in\widehat{\Xi}_{N}$, the cycle
generator on $\gamma_{\bs{x}}^{N}$ given by
\begin{equation*}
(\mc{L}_{N,\bs{x}}f)(\bs{x}+\bs{e}_{i}^{N})\;=\;
\widetilde{R}_{N}(\bs{x}+\bs{e}_{i}^{N},\bs{x}+\bs{e}_{i+1}^{N})
\left[f(\bs{x}+\bs{e}_{i+1}^{N})-f(\bs{x}+\bs{e}_{i}^{N})\right]\;,
\end{equation*}
for $0\le i\le 2$, where the jump rate $\widetilde{R}_{N}$ is given by
\begin{equation}
\label{jump1}
\begin{aligned}
\widetilde{R}_{N}(\bs{x}+\bs{e}_{i}^{N},\bs{x}+\bs{e}_{i+1}^{N})
\;=\;\exp\{-\beta N[\,\overline{F}_{\beta,N}(\bs{x})-
F_{\beta,N}(\bs{x}+\bs{e}_{i}^{N})]\} \;,
\end{aligned}
\end{equation}
and
\begin{equation*}
\overline{F}_{\beta,N}(\bs{x}) \;=\; \frac 13
\big\{F_{\beta,N}(\bs{x}+\bs{e}_{0}^{N})
+F_{\beta,N}(\bs{x}+\bs{e}_{1}^{N})
+F_{\beta,N}(\bs{x}+\bs{e}_{2}^{N}) \big\}\;.
\end{equation*}
An elementary computation shows that for $\bs{x}\in\widehat{\Xi}_{N}$,
$0\le i\le 2$,
\begin{equation}
\label{jump2}
R_{N}(\bs{x}+\bs{e}_{i}^{N},\bs{x}+\bs{e}_{i+1}^{N})
\;=\;w_{N}(\bs{x})\, \widetilde{R}_{N}(\bs{x}+\bs{e}_{i}^{N},
\bs{x}+\bs{e}_{i+1}^{N})\;,
\end{equation}
where the weights $w_{N}(\bs{x})$, $\bs{x}\in\widehat{\Xi}_N$, are given by
\begin{equation}
\label{weight}
w_{N}(\bs{x})\;=\;\left[x_{0}\left(x_{1}+\frac{1}{N}\right)
\left(x_{2}+\frac{1}{N}\right)\right]^{\frac{1}{3}}\;.
\end{equation}
Hence, the generator $\mc{L}_{N}$ defined in (\ref{gen3}) can be
represented in terms of the cycle generators $\mc{L}_{N,\bs{x}}$, $\bs{x}\in\widehat{\Xi}_N$, as
\begin{equation}
\label{cd1}
\mc{L}_{N}\;=\;\sum_{\bs{x}\in \widehat{\Xi}_N} w_{N}(\bs{x})\,
\mc{L}_{N,\bs{x}}\;. 
\end{equation}
Note that the weight function $w_{N}(\bs{x})$ converges to
$w(\bs{x})=(x_{0}x_{1}x_{2})^{1/3}$ uniformly on every compact subsets
of $\mbox{int}(\Xi)$.  Thereby, the Markov chain $\bs{r}_{N}(t)$ is a
special case of the model considered in Remark 2.9 of \cite{LS}.

\smallskip \noindent\textbf{The potential $F_{\beta}$.} The global
structure of the inter-valley dynamics is essentially related to the
potential $F_{\beta}$ defined in (\ref{pot2}). 

Denote by $\partial_{x_i} F_\beta$, $i=1,\,2$, the partial derivative of
$F_{\beta}$  with respect to $x_i$. We have that
\begin{equation}
\label{der1}
\begin{aligned}
 &   (\partial_{x_1}F_{\beta})(\bs{x})\;=\;-\, \frac{3}{2}(x_{1}-x_{0}) \;+\;
\frac{1}{\beta}\log\frac{x_{1}}{x_{0}} \;-\;
r_{\text{e}}\left[\cos\left(\theta_{\text{e}} - \frac{2\pi}{3}\right) -
\cos\theta_{\text{e}}\right]\;, \\
 &   (\partial_{x_2}F_{\beta})(\bs{x})\;=\;-\, \frac{3}{2}(x_{2}-x_{0})
 \;+\; \frac{1}{\beta}\log\frac{x_{2}}{x_{0}} \;-\; r_{\text{e}}
\left[\cos\left(\theta_{\text{e}}-\frac{4\pi}{3}\right)-
\cos\theta_{\text{e}}\right]\;.
\end{aligned}
\end{equation}
Therefore, a point $\bs x \in \Xi$ is a critical point of $F_\beta$ if
and only if
\begin{equation}
\label{critical}
\frac{1}{\beta}\log x_{k} \;-\; \frac{3}{2}x_{k} \;-\; 
r_{\text{e}}\, \cos\left(\theta_{\text{e}}-\frac{2k\pi}{3}\right)
\end{equation}
are equal for $k=0$, $1$, $2$.

The Hessian of $F_{\beta}$, denoted by $\nabla^{2}F_{\beta}$, is given by
\begin{equation}
\label{hess}
(\nabla^{2}F_{\beta}) (\bs{x})\;=\;
\begin{pmatrix}\frac{1}{\beta x_{0}}+\frac{1}{\beta x_{1}}-3 & 
\frac{1}{\beta x_{0}}-\frac{3}{2}\\
\frac{1}{\beta x_{0}}-\frac{3}{2} & 
\frac{1}{\beta x_{0}}+\frac{1}{\beta x_{2}}-3
\end{pmatrix}\;.
\end{equation}

\section{Zero External Magnetic Field}
\label{sec4}

We examine in this section the metastable behavior of the Potts model
under the assumption that the magnetic field vanishes:
$\bs{h}_{\text{e}}=\bs{0}$.

\subsection{Structure of Valleys}
\label{sec41}

To describe the valleys of the potential $F_{\beta}$, we first
identify in Proposition \ref{pro1} below all critical points of
$F_{\beta}$. 

For a point $\bs{x}$ at the boundary of $\Xi$, let $\bs{n}(\bs{x})$ be
the exterior normal vector at $\bs{x}$ with respect to the domain
$\Xi$. By \eqref{der1}, 
\begin{equation}
\label{01f}
\nabla F_\beta (\bs{x})\cdot \bs{n}(\bs{x}) \;=\;\infty \;, 
\end{equation}
with the convention that $\log 0 =-\infty$. In particular,
$F_\beta$ does not have minima at the boundary and the global minimum
is attained in the interior of $\Xi$, at some local minima.

According to the condition (\ref{critical}), a point
$\bs{x}$ is a critical point of $F_{\beta}$ if and only if
\begin{equation}
\label{cr1}
\frac{1}{\beta} \log x_{0}-\frac{3}{2} x_{0}\;=\;
\frac{1}{\beta} \log x_{1}-\frac{3}{2} x_{1}\;=\;
\frac{1}{\beta} \log x_{2}-\frac{3}{2} x_{2}\;.
\end{equation}
In particular, $\bs{p}=(1/3,\,1/3)$ is a critical point, which
corresponds to the configuration in which one third of the sites takes
the value $\bs v_k$ for $k=0$, $1$, $2$. This point is stable only at
high temperature, when the entropy plays an important role. This is
the content of the next lemma.

\begin{lemma}
\label{prop1}
The point $\bs{p}$ is a local minima of $F_{\beta}$ for $\beta<\beta_1
:= 2$, and a local maxima of $F$ for $\beta>\beta_1$.
\end{lemma}

\begin{proof}
We have already seen that $\bs{p}$ is a critical point of
$F_{\beta}$ regardless of $\beta$. By (\ref{hess})
the Hessian of $F$ at $\bs{p}$ is given by 
\begin{equation}
\label{hessp}
(\nabla^{2}F_{\beta}) (\bs{p})\;=\;
\frac{3(2-\beta)}{2\beta}\,
\begin{pmatrix}
2 & 1\\
1 & 2
\end{pmatrix}\;.
\end{equation}
The statement of the lemma follows from this expression. 
\end{proof}

Clearly, for each fixed $\beta>0$, $k>0$, the equation
$\frac{1}{\beta}\log x -\frac{3}{2} x=k$ has at most two positive real
solutions. Therefore, any point $\bs x = (x_0,x_1,x_2)$ which
satisfies (\ref{cr1}) must have two equal coordinates. Let $t$ be the
common value of two coordinates. Since the total sum is $1$, $t<1/2$
and the third value is $1-2t$. By (\ref{cr1}), $t$ satisfies the
equation
\begin{equation*}
\frac{1}{\beta} \log t -\frac{3}{2} t\;=\;\frac{1}{\beta} \log(1-2t)-\frac{3}{2}(1-2t)\;.
\end{equation*}
This equation can be rewritten as $f_{0}(t)=\beta$, where the function
$f_{0}:(0,1/2)\rightarrow\mathbb{R}$ is defined by
\begin{equation*}
f_{0}(t)\;=\;\begin{cases}
\frac{2}{3(1-3t)}\log\frac{1-2t}{t} & \mbox{if \ensuremath{t\neq1/3}}\\
2 & \mbox{if }t=1/3\;.
\end{cases}
\end{equation*}
The graph of $f_{0}$ is presented in Figure \ref{fig2}. Denote by
$m_{0}$ the point at which $f_{0}$ achieves its minimum, and let
$\beta_{3} =f_{0}(m_{0})$. The respective numerical values of $m_{0}$
and $\beta_{3}$ are approximately $0.2076$ and $1.8304$.

\begin{figure}
\includegraphics[scale=0.051]{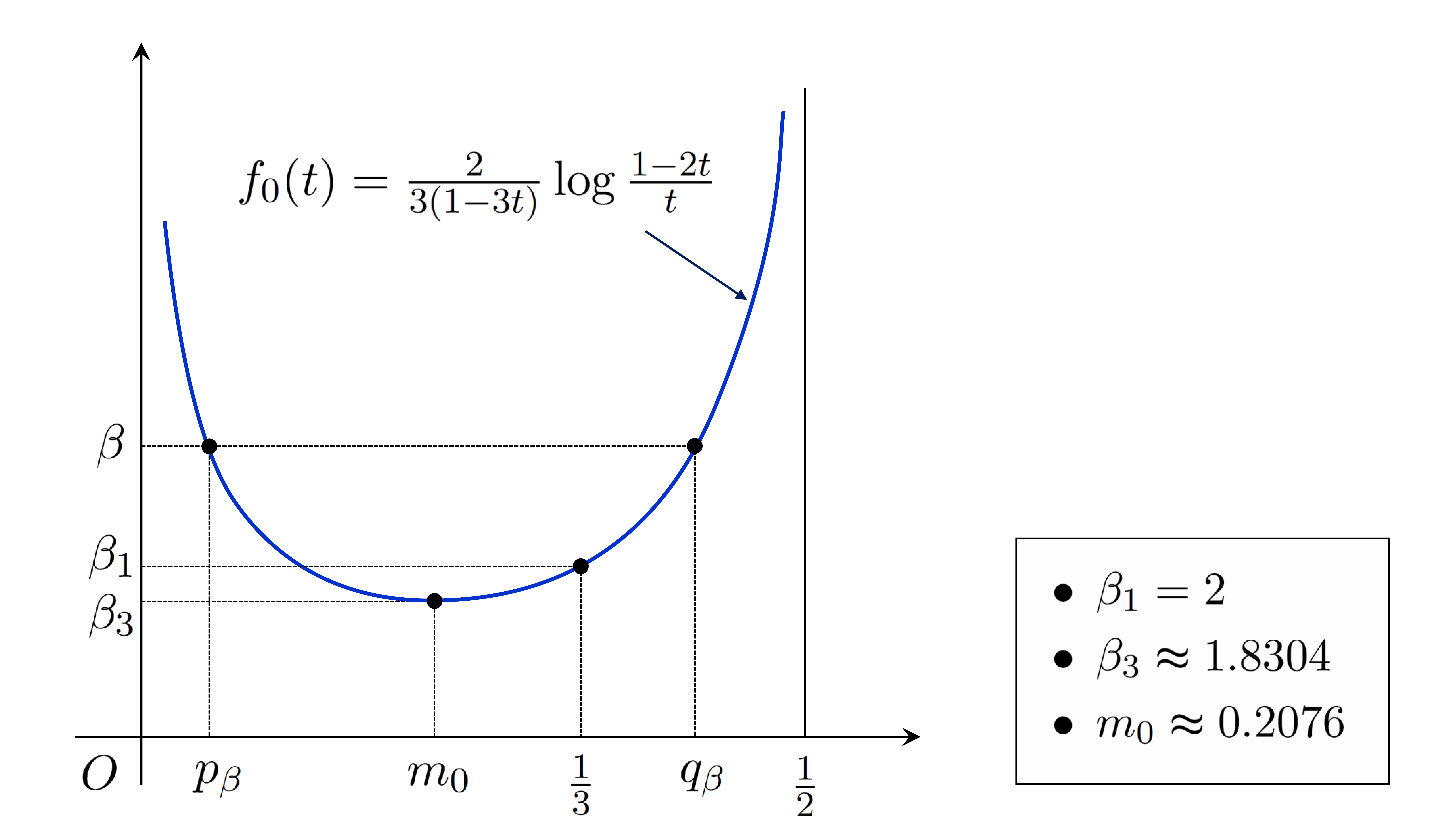}\protect
\caption{\label{fig2}The graph of $f_{0}(t)=\frac{2}{3(1-3t)}\log\frac{1-2t}{t}$
and the critical temperature $\beta_{3}=f(m_{0})$. The equation $f_{0}(t)=\beta$
has two solutions $p_{\beta},\,q_{\beta}$ provided that $\beta>\beta_{3}$,
where $p_{\beta}<m_{0}<q_{\beta}$. }
\end{figure}

By definition of $\beta_3$, for $\beta<\beta_3$, the equation
$f_0(t)=\beta$ has no solutions. On the other hand, for
$\beta>\beta_{3}$, this equation has two solutions, denoted by
$p_{\beta}<m_{0}<q_{\beta}$. For $\beta=\beta_3$, define
$p_\beta=q_\beta=m_0$.  In consequence, the triples
$(1-2p_{\beta},p_{\beta},p_{\beta})$,
$(1-2q_{\beta},q_{\beta},q_{\beta})$, and all triples obtained from
these two by permuting the coordinates, solve the equation
\eqref{cr1}.

These points correspond to critical points of $F_{\beta}$. For
$\beta>\beta_{3}$, let
\begin{equation}
\label{01}
\begin{aligned}
&  \bs{m}_{0}^{\beta}\;=\;(p_{\beta},p_{\beta})\;, \quad
\bs{m}_{1}^{\beta}\;=\;(1-2p_{\beta},p_{\beta})\;, \quad
\bs{m}_{2}^{\beta}\;=\;(p_{\beta},1-2p_{\beta})\;,\\
&\quad \bs{\sigma}_{0}^{\beta}\;=\;(q_{\beta},q_{\beta})\;, \quad
\bs{\sigma}_{1}^{\beta}\;=\;(1-2q_{\beta},q_{\beta})\;, \quad
\bs{\sigma}_{2}^{\beta}\;=\;(q_{\beta},1-2q_{\beta})\;.
\end{aligned}
\end{equation}
For $\beta=\beta_3$, $\bs{m}_i^\beta = \bs{\sigma}_i^\beta$ for $0\le
i\le2$, and for $\beta=\beta_{1}$, where $\beta_1=2$ has been
introduced in Lemma \ref{prop1}, $q_{\beta}=1/3$ so that
$\bs{\sigma}_{i}^{\beta} = \bs{p}$ for $0\le i\le2$. Up to this point, we figured out all the possible critical points of $F_\beta$ for all $\beta>0$.  

Let $\bs l_i(t)$, $0\le t\le 1/2$, $0\le i\le 2$, be the line given by
\begin{equation}
\label{line1}
\bs l_0(t)=(t,t)\;, \quad \bs l_1(t)=(1-2t,t)\;,
\quad \bs l_2(t)=(t,1-2t)\;.
\end{equation}
These lines correspond to the sets $\{\bs x \in \Xi : x_1=x_2\}$,
$\{\bs x\in \Xi : x_2=x_0\}$ and $\{\bs x \in \Xi: x_0=x_1\}$,
respectively. An elementary computation shows that
\begin{equation}
\label{pq1}
\frac{\textup{d}}{\textup{d}t} F_\beta (\bs{l}_i(t))\;=\; \frac
{3}{\beta} \, (3t-1)\, (f_0 (t)- \beta)
\end{equation}
for $0\le i\le 2$ and $0\le t\le 1/2$. 

A critical point of the potential $F_\beta$ is said to
\textit{degenerate} if the determinant of the Hessian of $F_\beta$ at
that critical point vanishes. Next result characterizes all critical
points of $F_{\beta}$.

\begin{proposition}
\label{pro1}
The critical points of $F_{\beta}$ are given by
\begin{enumerate}
\item For $\beta\in(0,\beta_{3})$, $\bs{p}$ is the unique critical
  point, and $\bs p$ is the global minima.
\item For $\beta=\beta_3$, $\bs{m}_0^\beta$, $\bs{m}_1^\beta$,
  $\bs{m}_2^\beta$ and $\bs{p}$ are the unique critical points.  The
  first three points are degenerate critical points which are not local
  minima, while $\bs{p}$ is the the global minimum.
\item For $\beta\in(\beta_{3},\beta_{1})$, the unique critical points
  are the four local minima $\bs{m}_{0}^{\beta}$, $\bs{m}_{1}^{\beta}$,
  $\bs{m}_{2}^{\beta}$, $\bs{p}$ and the three saddle points
  $\bs{\sigma}_{0}^{\beta},\,\bs{\sigma}_{1}^{\beta},\,\bs{\sigma}_{2}^{\beta}$.
\item For $\beta=\beta_1$, $\bs{m}_0^\beta$, $\bs{m}_1^\beta$,
  $\bs{m}_2^\beta$ and $\bs{p}$ are the unique critical points. The
  first three points are local minima, and $\bs{p}$ is a degenerate
  critical point which is not a local minima.
\item For $\beta\in(\beta_{1},\infty)$, the unique critical points are
  the three global minima $\bs{m}_{0}^{\beta}$, $\bs{m}_{1}^{\beta}$,
  $\bs{m}_{2}^{\beta}$, the three saddle points
  $\bs{\sigma}_{0}^{\beta}$, $\bs{\sigma}_{1}^{\beta}$,
  $\bs{\sigma}_{2}^{\beta}$, and the local maximum $\bs p$.
\end{enumerate}
\end{proposition}

\begin{proof}
Since there is no solution of $f_{0}(t)=\beta$ for $\beta\in(0,\beta_{3})$,
in this temperature range the unique critical point of $F_{\beta}$ is
$\bs{p}$, which is the global minimum of $F_{\beta}$, as claimed in (1).

Assume that $\beta=\beta_3$. By Lemma \ref{prop1}, $\bs{p}$ is a local
minima, and, by the observation next to \eqref{01}, $\bs{m}_i^\beta =
\bs{\sigma}_i^\beta$ for $0\le i\le2$. It remains, therefore, to check
that the critical points $\bs{m}_i^\beta$, $0\le i\le 2$, are
degenerate and are not local minima. 

By \eqref{hess}, the determinant of the Hessian of $F_\beta$ at these
points can be represented as a function of $m_0$, and the degeneracy
is easily shown by using the fact that $m_0$ solves the equation
$f_0'(m_0)=0$.

We now prove that the points $\bs{m}_i^\beta$, $0\le i\le 2$, are not
local minima. Consider the value of $F_\beta$ restricted to the line
$\bs{l}_i(t)$ introduced in \eqref{line1}.  By \eqref{pq1}, since
$m_0<1/3$, $(d/dt) F_\beta(\bs{l}_i(t)) <0$ for $t$ in a neighborhood
of $m_0$, $t\not = m_0$. In particular, $m_0$ is not a local minimum
of $F_\beta(\bs{l}_i(t))$, which proves that $\bs{m}_i^\beta =
\bs{l}_i (m_0)$ is not a local minima of $F_\beta$.  

Assume that $\beta>\beta_3$, $\beta\not = \beta_1$. In view of Lemma
\ref{prop1}, to prove claims (3) and (5), it is enough to show that
the points $\bs{m}_{i}^{\beta}$, $0\le i\le2$, are local minima, and
that the points $\bs{\sigma}_{i}^{\beta}$, $0\le i\le2$, are saddle
points.

Fix $\bs{m}_{i}^{\beta}$, $0\le i\le2$.  We claim that
\begin{equation}
\label{pp1}
\frac{1}{\beta p_{\beta}}\;>\;\frac{3}{2}\;\;\mbox{\;and}\;\;\;
\frac{1}{\beta p_{\beta}}+\frac{2}{\beta(1-2p_{\beta})}\;>\;
\frac{9}{2}\;\cdot
\end{equation}
Replacing $\beta$ by $f_{0}(p_{\beta})$ in the first inequality, it
becomes
\begin{equation*}
\log\frac{1-2p_{\beta}}{p_{\beta}}\;<\;\frac{1-3p_{\beta}}{p_{\beta}}\;,
\end{equation*}
which follows from the elementary inequality $\log x<x-1$ for
$x\neq1$.  For the second inequality of (\ref{pp1}), replace $\beta$
by $f_{0}(p_{\beta})$ to rewrite the inequality as $g_{0}
(p_{\beta})>0$, where
\begin{equation}
\label{g1}
g_{0}(t)\;=\;\left(\log t+\frac{1}{3t}\right) \;-\;
\left(\log(1-2t)+\frac{1}{3(1-2t)}\right)\;.
\end{equation}
The function $g_{0}$ is decreasing in the interval $(0,\,1/4)$ since
\begin{equation}
\label{g2}
g_{0}'(t)\;=\;-\, \frac{(3t-1)(4t-1)}{3t^{2}(1-2t)^{2}}
\end{equation}
is negative in this interval.  Recall that $m_0$ is the point at which
$f_0$ achieves its minimum.  Since $p_{\beta}<m_{0}$, we have that
$g(p_{\beta})>g(m_0)$. From the condition $f_0'(m_0)=0$, it is easy to check that $g_0(m_0)=0$ and this proves \eqref{pp1}.

We turn to the Hessian at $\bs{m}_{i}^{\beta}$, $0\le i\le2$. The
determinant of $(\nabla^{2}F_{\beta})(\bs{m}_{i}^{\beta})$ can be
written as
\begin{equation*}
\left(\frac{1}{\beta p_{\beta}}-\frac{3}{2}\right)
\left(\frac{1}{\beta p_{\beta}}+\frac{2}{\beta(1-2p_{\beta})}-
\frac{9}{2}\right)\;,
\end{equation*}
which is positive by (\ref{pp1}). By similar reasons the trace of
$(\nabla^{2}F_{\beta})(\bs{m}_{i}^{\beta})$ is positive. In
particular, the points $\bs{m}_{0}^{\beta}$, $\bs{m}_{1}^{\beta}$, and
$\bs{m}_{2}^{\beta}$ are local minima.

Consider a critical point $\bs{\sigma}_{i}^{\beta}$, $0\le i\le2$. We
claim that
\begin{align}
& \frac{1}{\beta q_{\beta}}\;<\;\frac{3}{2}\;\;\mbox{\;and}
\;\;\;\frac{1}{\beta q_{\beta}}+
\frac{2}{\beta(1-2q_{\beta})}\;>\;\frac{9}{2}
\mbox{\;\;\;for\;}\,\beta>\beta_{1}\;,
\label{qq1}\\
& \quad\frac{1}{\beta q_{\beta}}\;>\;\frac{3}{2}\;\;\;\mbox{and}
\;\;\;\frac{1}{\beta q_{\beta}}+\frac{2}{\beta(1-2q_{\beta})}
\;<\;\frac{9}{2}\;\;\;\mbox{for\;}\,\beta<\beta_{1}\;.
\label{qq2}
\end{align}
The first inequality in (\ref{qq1}) follows from the fact that
$\beta>2$ and $q_{\beta}>1/3$ (cf. Figure \ref{fig2}). For the second
inequality, replace $\beta$ by $f_{0}(q_{\beta})$ as before. Since
$q_{\beta}>1/3$, we can rewrite the inequality as
$g_{0}(q_{\beta})<0$, where $g_{0}$ is defined in (\ref{g1}). By
(\ref{g2}), $g_{0}(t)$ is decreasing for $t>1/3$ so that
$g(q_{\beta})<g(1/3)=0$. The proof for (\ref{qq2}) is analogous.

By (\ref{qq1}) and (\ref{qq2}), the determinant of
$(\nabla^{2}F_{\beta})(\bs{\sigma}_{i}^{\beta})$, $0\le i\le2$, which
is equal to
\begin{equation*}
\left(\frac{1}{\beta q_{\beta}} - \frac{3}{2}\right)\,
\left(\frac{1}{\beta q_{\beta}}+\frac{2}{\beta(1-2q_{\beta})}
-\frac{9}{2}\right)\;,
\end{equation*}
is negative. This completes the proof.

Finally, assume that $\beta=\beta_1$. The proof presented to show that
the points $\bs{m}_{i}^{\beta}$, $0\le i\le2$, are local minima of
$F_{\beta}$ in the case $\beta>\beta_3$ is in force for
$\beta=\beta_1$. On the other hand, we can show that $\bs p$ is a
degenerate critical point which is not a local minimum as we proved
that points $\bs{m}_{i}^{\beta}$ have these properties in the case
$\beta=\beta_3$.
\end{proof}

It follows from the definition of $f_0$ that $\lim_{\beta\to\infty}
p_{\beta} = 0$ (cf.  Figure \ref{fig2}). Hence, the local minima
$\bs{m}_{i}^{\beta}$ corresponds to the configurations in which most
of the spins are aligned with $\bs{v}_{i}$. 

According to Proposition \ref{pro1}, the point $\bs{p}$ is a global
attractor if $\beta<\beta_{3}$, and the critical points
$\bs{m}_{i}^{\beta}$, $0\le i\le 2$, are the unique stable equilibria
if $\beta>\beta_{1}$. In the range $(\beta_{3},\beta_{1})$ these $4$
local minima coexist. We examine more closely this case.

Since $F_{\beta}$ is symmetric with respect to $x_{0}$, $x_{1}$,
$x_{2}$, the quantities $H_\beta$ and $h_\beta$ introduced below are
well defined:
\begin{align*}
& H_{\beta}\;=\;F_{\beta}(\bs{\sigma}_{0}^{\beta})
\;=\;F_{\beta}(\bs{\sigma}_{1}^{\beta})
\;=\;F_{\beta}(\bs{\sigma}_{2}^{\beta})\;,\\
& \quad h_{\beta}\;=\;F_{\beta}(\bs{m}_{0}^{\beta})
\;=\;F_{\beta}(\bs{m}_{1}^{\beta})\;=\;F_{\beta}(\bs{m}_{2}^{\beta})\;.
\end{align*}
Recall that $F_{\beta}(\bs{p})=0$ for all $\beta>0$. 

\begin{lemma}
\label{lem01}
There exists $\beta_{2}\in(\beta_{3},\beta_{1})$ such that
$h_{\beta}>0$ for $\beta\in(\beta_{3},\beta_{2})$, $h_{\beta_{2}}=0$,
and $h_{\beta}<0$ for $\beta\in(\beta_{2},\beta_{1})$.
\end{lemma}

\begin{proof}
By (\ref{pot2}), we can write $h_{\beta}=F_{\beta}(\bs{m}_{i}^{\beta})$
as
\begin{equation*}
h_{\beta}\;=\;-\frac{1}{2}(1-3p_{\beta})^{2}+
\frac{1}{\beta}\Big\{ 2p_{\beta} \, \log (3p_{\beta}) 
\,+\, (1-2p_{\beta})\, \log \big[3(1-2p_{\beta}) \big] \Big\} \;.
\end{equation*}
Replacing $\beta$ by $f_{0}(p_{\beta})$, the previous identity becomes 
\begin{equation*}
h_{\beta}\;=\;\frac{1-3p_{\beta}}{2\log[(1-2p_{\beta})/p_{\beta}]}
\, \Big\{ (2-3p_{\beta}) \log \big[3(1-2p_{\beta})\big] \,+\,
(3p_{\beta}+1) \log (3p_{\beta}) \Big \}\;.
\end{equation*}
Since $p_{\beta}<\frac{1}{4}$, $h_{\beta}$ has the same sign as
$k_{0}(p_{\beta})$, where
\begin{equation*}
k_{0}(t)\;=\;(2-3t)\log(1-2t)+(3t+1)\log t\;.
\end{equation*}
A straightforward computation gives that $k_{0}'(t)=3g_{0}(t)$, where
$g_{0}$ has been introduced in (\ref{g1}). In the proof of Proposition
\ref{pro1}, we proved that $g_{0}(t)>0$ for $t<m_0$, In particular,
$k_0 (t)$ is increasing for for $t<m_0$. To complete the proof it remains
to observe that $k_{0}(p_{\beta_{3}})>0>k_{0}(p_{\beta_{1}})$, that
$p_{\beta}$ is a continuous decreasing function of $\beta$ on
$(\beta_{3},\beta_{1})$, and to recall the intermediate value theorem.
\end{proof}

The approximate numerical value of $\beta_{2}$ is $1.8484$. By the
previous lemma, the global minima of $F_{\beta}$ is $\bs{p}$ for
$\beta<\beta_{2}$ and $\bs{m}_{0}^{\beta}$, $\bs{m}_{1}^{\beta}$,
$\bs{m}_{2}^{\beta}$ for $\beta>\beta_{2}$. For $\beta=\beta_{2}$,
these four points are global minima. This completes the description of
the metastable and the stable point of the potential $F_\beta$ in the
temperature regimes determined by the critical temperatures
$\beta_{3}<\beta_{2}<\beta_{1}$.  We refer to Figure \ref{fig3} for
the illustration of the characterization of three critical temperatures.

\begin{figure}
\includegraphics[scale=0.067]{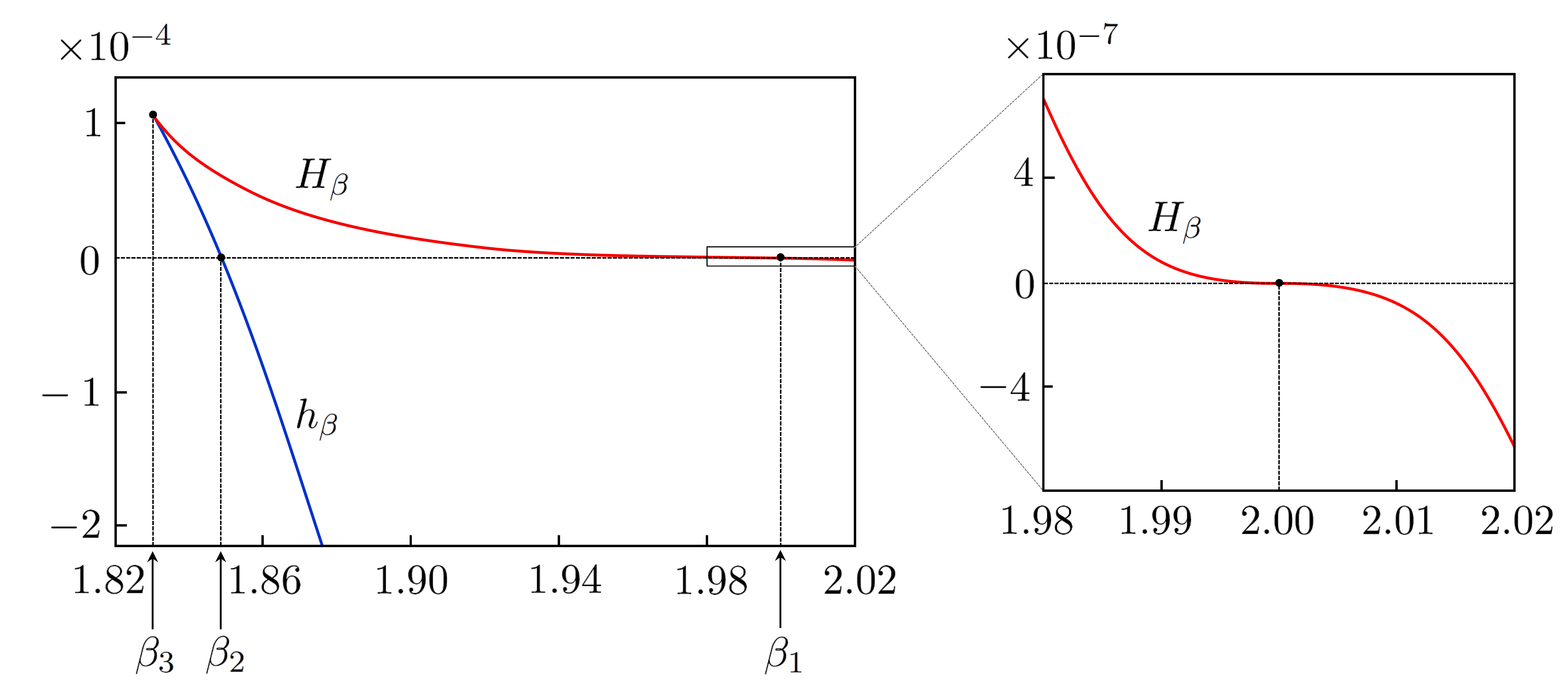}
\protect\caption{\label{fig3} The graph of $H_\beta$ and $h_\beta$ as
  functions of $\beta$, and the critical temperatures $\beta_1=2$,
  $\beta_2\thickapprox1.8484$ and $\beta_3\thickapprox1.8304$. For
  $\beta>\beta_3$, $\theta_\beta =H_\beta - h_\beta$ is the depth of
  the valleys $W_\beta(i), 0\le i\le 2$, and for
  $\beta\in(\beta_3,\beta_1)$, $H_\beta$ is the depth of the valley
  $W_\beta (3)$.}
\end{figure}
 
\subsection{Stable and Metastable Sets}

We introduce in this subsection some valleys around the local minima,
and we investigate the relationship between these sets and the saddle
points. Since it has been observed in the previous subsection that
there is no metastability behavior in the high temperature regime
$\beta \le \beta_{3}$, we assume that $\beta > \beta_{3}$.

\begin{figure}
  \protect
\includegraphics[scale=0.132]{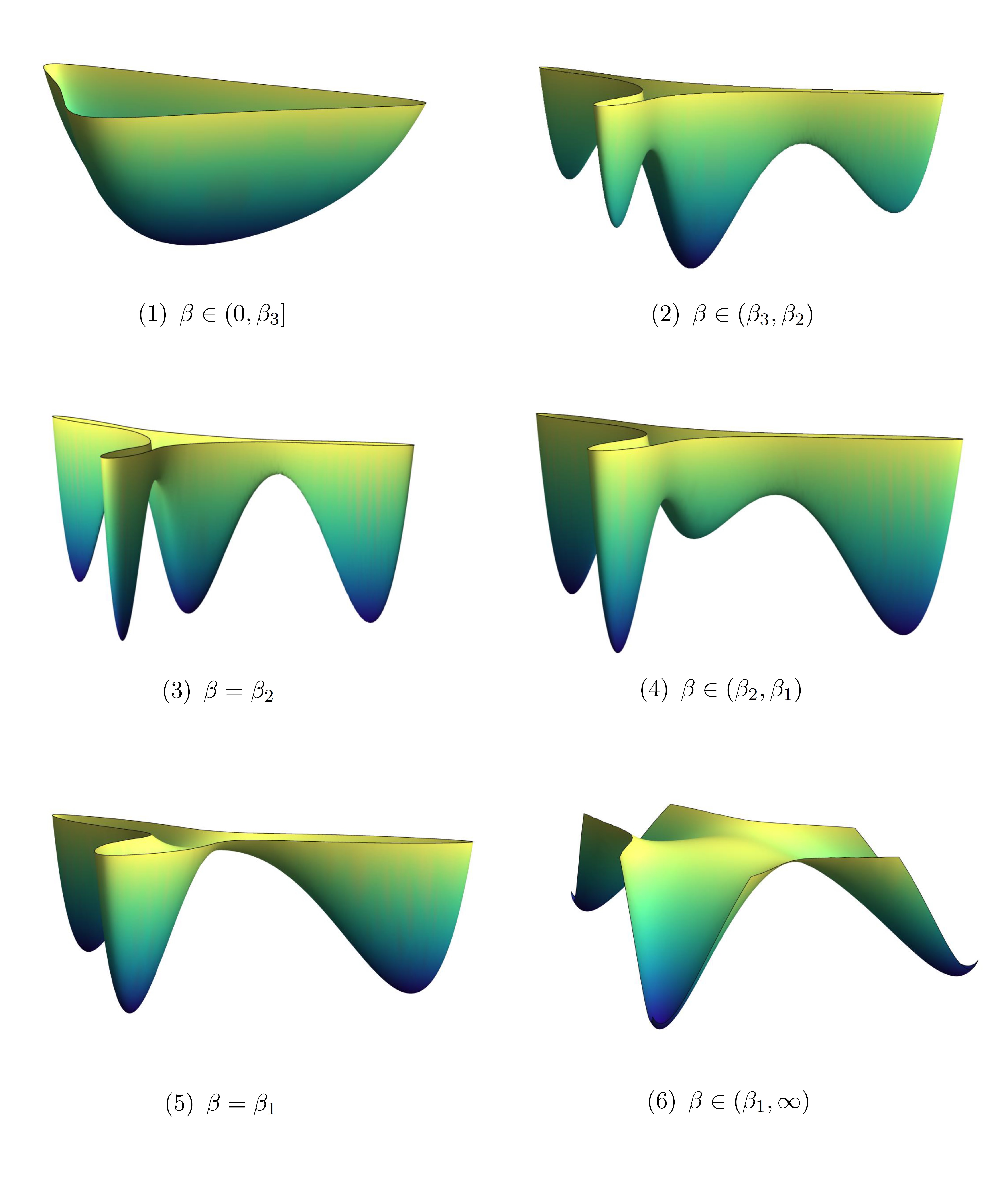}\protect
  \caption{\label{fig6}The graphs of $F_\beta (\bs{x})$ for various
temperature conditions. We used $\beta=1.6,\,1.843,\,1.86$ and $2.4$ for
(1), (2), (4) and (6), respectively.}
\end{figure}

Let $W_{\beta}=\{\bs x\in\Xi:F_{\beta}(\bs{x})<H_{\beta}\}$ and denote
by $W_{\beta}(i)$, $0\le i\le2$, the connected component of
$W_{\beta}$ containing $\bs{m}_{i}^{\beta}$. In addition, for
$\beta<\beta_1$, denote by $W_{\beta}(3)$ the connected component of
$W_\beta$ containing $\bs{p}$ (cf.  Figure \ref{fig4}). In the next
proposition we prove that in each temperature range the structure of
the valleys $W_\beta$ resembles the ones illustrated in Figure
\ref{fig4}.

\begin{figure}
\includegraphics[scale=0.050]{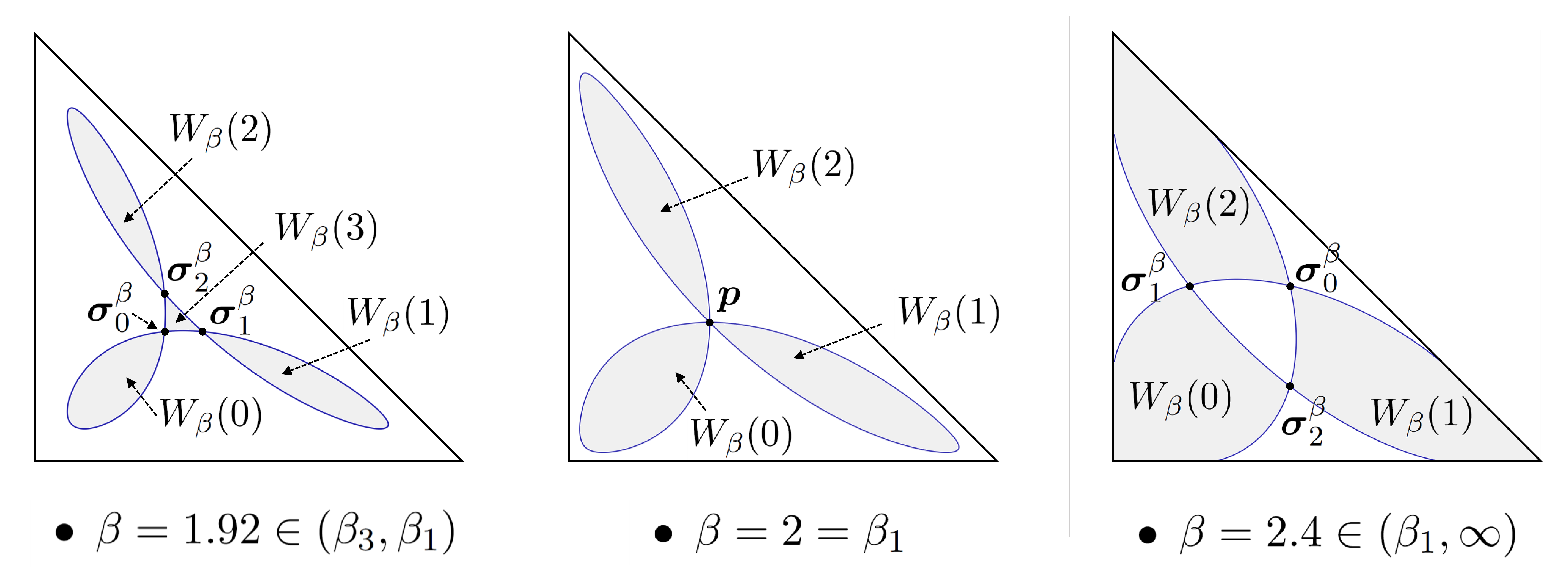}
\protect\caption{\label{fig4} Examples of valleys
  $W_{\beta}(0),\,W_{\beta}(1),\,W_{\beta}(2)$ and $W_{\beta}(3)$ at
  different temperatures. The blue contour denotes the level set
  $F_{\beta}^{-1}(H_\beta)$.}
\end{figure}

For $i\neq j$, let $\mf{S}_{i,j}=\overline{W_\beta(i)} \cap
\overline{W_\beta(j)}$, where $\overline{F}$ stands for the closure of
set $F\subset \mathbb R^2$.

\begin{proposition}
\label{prow} 
We have that
\begin{enumerate}
\item For $\beta\ge \beta_1$, the sets $W_\beta (i)$, $0\le i\le 2$,
  are different.
\item For $\beta>\beta_1$, $\mf{S}_{i,j}=\{\bs{\sigma}_k^\beta\}$,
  where $\{i,j,k\} = \{0,1,2\}$.
\item For $\beta=\beta_1$ and $0\le i \neq j\le 2$,
  $\mf{S}_{i,j}=\{\bs{p}\}$.
\item For $\beta_3<\beta<\beta_1$, the sets $W_\beta (i)$, $0\le i\le 3$, are
  different.
\item For $\beta_3<\beta<\beta_1$, $\mf{S}_{i,j}$ is empty for
  $0\le i\neq j\le 2$, and $\mf{S}_{i,3}=\{\bs{\sigma}_i^\beta\}$ for
  $0\le i\le2$.
\end{enumerate}
\end{proposition}

\begin{proof} 
Fix $\beta \ge \beta_1$.  To prove that the sets $W_\beta(i)$, $0\le
i\le 2$, are different, recall from \eqref{line1} the definitions of
the lines $\bs{l}_i (t)$, and set $\bs{q}_i = \bs{l}_i (1/2)$. 
Each segment $\overline{\bs{p}\bs{q}_i}$, $0\le i\le 2$, can be
represented as
\begin{equation*}
\overline{\bs{p}\bs{q}_i}\;=\;\big\{\bs{l}_i (t): 1/3\le t\le
1/2\big\}\;. 
\end{equation*}
The segments $\overline{\bs{p}\bs{q}_i}$, $0\le i\le 2$, divide the
set $\Xi$ into three pieces, denoted by $\Xi^{(i)}$, $0\le i\le 2$, such that $\bs{m}_i^\beta \in \Xi^{(i)}$ (cf. Figure
\ref{fig5}). By \eqref{pq1}, the potential $F_\beta$ restricted to
each of these segments attains its minimum at $\bs{\sigma}_i^\beta$,
and hence $F_\beta(\bs{x})\ge H_\beta$ for all $\bs{x}\in
\overline{\bs{p}\bs{q}_i}$, $0\le i\le 2$. This proves that
$W_\beta(i) \subset \Xi^{(i)}$ for $0\le i\le 2$. In particular,
the valleys $W_\beta (i)$ are all different, as asserted in (1)

Assume that $\beta > \beta_1$. Since $\bs{\sigma}_i^\beta$ is a saddle
point, there is an eigenvector, represented by $\bs{w}_i$,
corresponding to the negative eigenvalue of the Hessian of $F_\beta$
at $\bs{\sigma}_i^\beta$. Hence, the function $t \mapsto F_\beta
(\bs{\sigma}_i^\beta + t \bs{w}_i)$ achieves a local maximum at
$t=0$. Let $\epsilon>0$ be a small number such that $F_\beta
(\bs{\sigma}_0^\beta + t \bs{w}_0)<H_\beta$ for all $0<|t|<\epsilon$.
Assume, without loss of generality, that $\bs{\sigma}_0^\beta +
\epsilon \bs{w}_0\in W_\beta(1)$ and that $\bs{\sigma}_0^\beta -
\epsilon \bs{w}_0 \in W_\beta(2)$. Consider the path $\{\bs y(t) :
t\ge 0\}$ described by the ordinary differential equation
\begin{equation*}
\dot{\bs{y}}(t)\;=\;-\nabla F_\beta \big(\bs{y}(t)\big)
\;, \quad \bs{y}(0)\;=\;\bs{\sigma}_0^\beta + \epsilon \bs{w}_0\;.
\end{equation*}
It is well known that $F_\beta (\bs{y}(t))$ is a decreasing function
of $t$ and that $\bs{y}(t)$ converges to a local minimum of $F_\beta$
as $t\uparrow \infty$. Since $F_\beta(\bs{y}(0))<H_\beta$, this path
cannot cross the segments $\overline{\bs{p}\bs{q}_i}$, $0\le i\le 2$, and,
by \eqref{01f}, it can not hit the boundary of $\Xi$.  It stays,
therefore, in the interior of $\Xi^{(1)}$ for all $t\ge 0$.  Since
$\bs{m}_1^\beta$ is the unique critical point of $F_\beta$ in the
interior of $\Xi^{(1)}$, $\bs{y}(t)$ must converge to $\bs{m}_1^\beta$
as $t\uparrow \infty$. This proves that $\bs{\sigma}_0^\beta$ and
$\bs{m}_1^\beta$ are connected by a continuous path, along which
$F_\beta$ is less than $H_\beta$, except at $\bs{\sigma}_0^\beta$. In
particular, $\bs{\sigma}_0^\beta$ belongs to $\overline{W_\beta
  (1)}$. Similarly, $\bs{\sigma}_0^\beta \in \overline{W_\beta (2)}$,
so that $\{\bs{\sigma}_0^\beta\} \subset \mathfrak{S}_{1,2}$. 

To prove the inverse relation, note that by the first assertion of the
proposition and by the definition of the segment
$\overline{\bs{p}\bs{q}_0}$, $\mathfrak{S}_{1,2} = \overline{W_\beta(1)}
\cap \overline{W_\beta(2)} \subset \Xi^{(1)} \cap \Xi^{(2)} =
\overline{\bs{p}\bs{q}_0}$. Since $\bs{\sigma}_0^\beta $ is the only point
$\bs x$ in the segment $\overline{\bs{p}\bs{q}_0}$ such that $F_\beta(\bs
x)= H_\beta$, $\mathfrak{S}_{1,2} \subset \{\bs{\sigma}_0^\beta\}$, so
that $\mathfrak{S}_{1,2} = \{\bs{\sigma}_0^\beta\}$

The same argument shows that $\mathfrak{S}_{i,j} =
\{\bs{\sigma}_k^\beta\}$ for all $\{i,\,j,\,k\}=\{0,\,1,\,2\}$. This
completes the proof of (2).

Assume that $\beta=\beta_1$. By \eqref{pq1}, $F_\beta (\bs{l}_i (t))<
H_\beta =0$ for all $t\in [p_\beta ,1/3)$, and $F_\beta(\bs{l}_i
(1/3))=F_\beta(\bs{p})=0$. At this point we may repeat the arguments
presented in the case $\beta>\beta_1$ to conclude that $\mf{S}_{i,j} =
\{\bs{p}\}$, as asserted in (3).

Assume that $\beta_3 < \beta <\beta_1$. Let $\bs{k}_i(t)$,
$t\in [-q_\beta,q_\beta]$, $0\le i\le2$, be the lines given by
\begin{equation*}
\bs{k}_0(t)\;=\;(q_\beta-t,q_\beta+t)\;,\;\;
\bs{k}_1(t)\;=\;(1-2q_\beta,q_\beta+t)\;,\;\;
\bs{k}_2(t)\;=\;(q_\beta-t,1-2q_\beta)\;.
\end{equation*}
The line $\bs{k}_i$ represents the set $\{\bs{x}:x_i=1- 2q_\beta\}$
and $\bs{k}_i(0) = \bs{\sigma}_i^\beta$. These three lines divide
$\Xi$ into four pieces if $q_\beta\ge 1/4$ and seven pieces if $q_\beta < 1/4$. For both of these cases, four of them contain exactly one of the points $\bs{m}_0^\beta$, $\bs{m}_1^\beta$, $\bs{m}_2^\beta$, $\bs{p}$
(cf. Figure \ref{fig5}). The function $F_\beta(\bs{k}_i(t))$ is
minimized at $t=0$, so that $F_\beta(\bs{k}_i(t))\ge F_\beta(\bs
\sigma_i) = H_\beta$ for all $t\in [-q_\beta, q_\beta]$. This proves
that $W_\beta (i)$, $0\le i\le 3$, are different sets, as stated in
(4).

The arguments presented in the proof of assertion (2) permit to show
that $\mf{S}_{i,3} = \{\bs{\sigma}_i^\beta\}$ for $0\le i\le 2$. On
the other hand, denote by $\Xi^{(j)}$, $0\le j\le 2$, the set which
contains the point $\bs m^\beta_j$ in the decomposition of $\Xi$ in
seven sets through the lines $\bs k_n(t)$ (cf. Figure \ref{fig5}). The intersection of
$\Xi^{(i)}$ with $\Xi^{(j)}$, $i\not = j$, is a singleton, and the
value of the potential $F_\beta$ at this point is larger than
$H_\beta$, which proves that $\mf{S}_{i,j} = \varnothing$ for $0\le
i, j\le 2$, $i\not = j$.
\end{proof}
 
\begin{figure}
\includegraphics[scale=0.060]{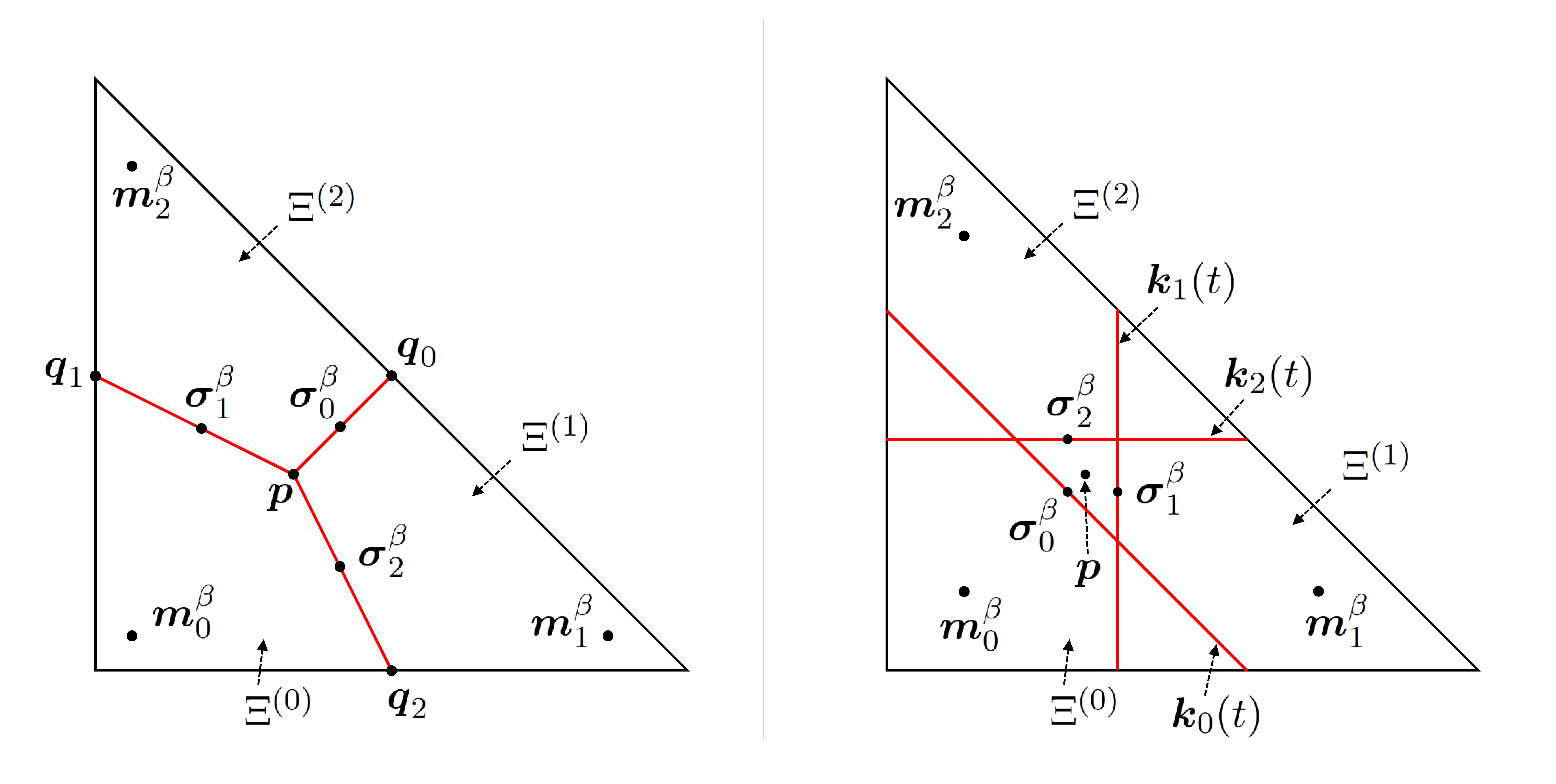}\protect
\protect\caption{\label{fig5}Visualizations of the proof of
  Proposition \ref{prow} for $\beta\in(\beta_1,\infty)$ (left) and
  $\beta \in (\beta_3,\beta_1)$ (right).}
\end{figure}

\subsection{Metastability result}

In this subsection, we present the metastable behavior of the chain
$\bs{r}_{N}(t)$ based on the results of \cite{LS}. We assume
throughout this section that $\beta>\beta_3$, 

Denote, from now on, $\bs{p}$ by $\bs{m}_{3}^{\beta}$, and define the
index set $\mf{I}_{\beta}$ by
\begin{equation}
\label{index}
\mf{I}_{\beta}\;=\;\{0,1,2\}\;\;\mbox{for }\beta>\beta_{1}
\quad \mbox{and}\quad \mf{I}_{\beta}\;=\;\{0,1,2,3\}\;\;
\mbox{for }\beta<\beta_{1}\;.
\end{equation}
Denote by $\theta_{\beta}(i)$, $i\in\mf{I}_{\beta}$, the
depth of the valley $W_{\beta}(i)$, namely,
\begin{equation}
\label{depth}
\begin{aligned}
& \theta_{\beta}(i)\;=\;\beta(H_{\beta}-h_{\beta})\;\;
\text{for $i=0$, $1$, $2$ and  $\beta>\beta_3$}\; , \\
& \qquad \theta_{\beta}(3)\;
=\;\beta H_{\beta} \text{ for } \beta_3<\beta<\beta_1\;.
\end{aligned}
\end{equation}
Let $\theta_{\beta} = \theta_{\beta}(0)= \theta_{\beta}(1)=
\theta_{\beta}(2)$, let $\epsilon$ be a small number satisfying
$0<\epsilon<\min_{i\in\mf{I}_{\beta}}\theta_{\beta}(i)$, and let
$W_{\beta}^{\epsilon}(i)\subset W_{\beta}(i)$, $i\in\mf{I}_{\beta}$,
be the connected component of
$\{\bs{x}\in\Xi:F_{\beta}(\bs{x})<H_{\beta}-\epsilon\}$.  The
metastable set $\mc{E}_{\beta}^{N}(i)$, $i\in\mf{I}_{\beta}$, is
defined as the discretization of $W_{\beta}^{\epsilon}(i)$:
$\mc{E}_{\beta}^{N}(i)=\Xi_{N}\cap W_{\beta}^{\epsilon}(i)$.

Let $\mathbb{A}$ (cf. \cite[display (2.7)]{LS}) be the
matrix given by
\begin{equation}
\label{matA}
\mathbb{A}\;=\;\sum_{i=0}^{2}(\bs{e}_{i}-\bs{e}_{i+1})
\, \bs{e}_{i}^{\dagger}\;=\;
\begin{pmatrix}
1 & 0\\
-1 & 1
\end{pmatrix}\;,
\end{equation}
where $\bs u^\dagger$ represents the transposition of the vector $\bs u$.
This matrix plays a significant role in the metastable behavior of
$\bs r_N(t)$, as observed in \cite{LS}.

\begin{lemma}
\label{lem3}
The determinant of $(\nabla^{2}F_{\beta})(\bs{x})$ and the
characteristic polynomial of $\mathbb{A} \cdot
(\nabla^{2}F_{\beta})(\bs{x})$ are symmetric with respect to
$x_{0},\,x_{1}$ and $x_{2}$.
\end{lemma}

\begin{proof}
Since
\begin{equation}
\label{dethess}
\det\left[(\nabla^{2}F_{\beta}) (\bs{x})\right]\;=\;
\frac{1}{\beta^{2}}\sum_{0\le i<j\le2}\frac{1}{x_{i}x_{j}} \;-\; 
\frac{3}{\beta}\sum_{i=0}^2 \frac{1}{x_{i}} \;+\; \frac{27}{4}\; ,
\end{equation}
the first assertion is in force. For the second one, since 
\begin{equation*}
\det\left[\mathbb{A}\cdot(\nabla^{2}F_{\beta})(\bs{x})\right]
\;=\;\det\mathbb{A}\cdot\det(\nabla^{2}F_{\beta})(\bs{x})
\;=\;\det(\nabla^{2}F_{\beta})(\bs{x})
\end{equation*}
is symmetric, it suffices to check the trace of
$\mathbb{A}\cdot(\nabla^{2}F_{\beta})(\bs{x})$ is symmetric. This is
obvious since
\begin{equation*}
\mbox{Tr}\left[\mathbb{A}\cdot (\nabla^{2}F_{\beta}) (\bs{x})\right]
\;=\;\frac{1}{\beta}\sum_{i=0}^2\frac{1}{x_{i}} \;-\; \frac{9}{2}\;\cdot
\end{equation*}
\end{proof}

Denote by $\nu(\bs{m}_{i}^{\beta})$, $i\in\mf{I}_{\beta}$, the
normalized asymptotic mass of the metastable set
$\mc{E}_{\beta}^{N}(i)$ (cf. \cite[display (2.9)]{LS}): 
\begin{equation}
\label{nu1}
\nu(\bs{m}_{i}^{\beta})\;=\;\lim_{N\rightarrow\infty}
\frac{\widehat{Z}_{N}(\beta)}{2\pi N} \, 
\exp\big\{ N\beta F_{\beta,N}(\bs{m}_{i}^{\beta})\big\} 
\, \nu_{\beta}^{N}(\mc{E}_{\beta}^{N}(i))\;.
\end{equation}
It is shown in Section 6 of \cite{LMT} that 
\begin{equation}
\label{nu2}
\nu(\bs{m}_{i}^{\beta}) \;=\; 
\frac{e^{-\beta  G_{\beta}(\bs{m}_{i}^{\beta})}}
{\sqrt{\beta^{2}\det\big[(\nabla^{2}F_{\beta})
(\bs{m}_{i}^{\beta})\big]}}\;\cdot
\end{equation}
Since, by Lemma \ref{lem3}, $\nu(\bs{m}_{0}^{\beta}) =
\nu(\bs{m}_{1}^{\beta}) = \nu(\bs{m}_{2}^{\beta})$, denote this value
by $\nu_{\beta}$, and let $\nu_{\beta}(3) := \nu(\bs{m}_{3}^{\beta})$.

Denote by $-\mu_{\beta}$ the negative eigenvalue of $\mathbb{A}\cdot
(\nabla^{2}F_{\beta}) (\bs{\sigma}_{\beta}^{i})$, $0\le i\le 2$. By
Lemma \ref{lem3}, this eigenvalue does not depend on $i$.  Denote by
$\omega_{\bs{\sigma}}$ the Eyring-Kramers constant of the saddle point
$\bs{\sigma}$ (cf. \cite[display (2.10) and Remark 2.9]{LS}):
\begin{equation}
\label{omega}
\omega(\bs{\sigma}_{i}^{\beta})\;=\;
e^{-\beta G_{\beta}(\bs{\sigma}_{i}^{\beta})} \, w(\bs{\sigma}_{i}^{\beta})\,
\frac{\mu_{\beta}}{\sqrt{-  \det\big[ (\nabla^{2}F_{\beta})
(\bs{\sigma}_{\beta}^{i})\big]}} \;.
\end{equation}
By Lemma \ref{lem3}, this quantity is independent of $i$. Hence,
let $\omega_{\beta}=\omega(\bs{\sigma}_{0}^{\beta})=
\omega(\bs{\sigma}_{1}^{\beta})=\omega(\bs{\sigma}_{2}^{\beta})$.

\smallskip\noindent\textbf{Regime I: $\beta\in(\beta_{1},\infty)$.} In
this range of temperatures, there are three valleys,
$\mc{E}_{\beta}^{N}(0)$, $\mc{E}_{\beta}^{N}(1)$ and
$\mc{E}_{\beta}^{N}(2)$, with same depth, and the process
$\bs{r}_{N}(t)$ exhibits a tunneling behavior between these three
valleys. The rigorous description can be stated as follows, in the
spirit of \cite{BL1,BL2,BL3}.

Define the projection map $\Psi_{N}: \Xi_{N} \rightarrow
\{0,\,1,\,2\}\cup\{N\}$ by
\begin{equation*}
\Psi_{N}(\bs{x})\;=\;\sum_{i=0}^{2} i\,\mathbf{1}
\{\bs{x}\in\mc{E}_{\beta}^{N}(i)\}
\;+\; N\,\mathbf{1}\left\{ \bs{x}\in \Delta_N \right\} \;,
\end{equation*}
where $\Delta_N = \Xi_{N} \setminus \cup_{0\le i\le 2}\,
\mc{E}_{\beta}^{N}(i)$. Let $\mathbb{X}_{N}(t)$ be the hidden Markov
chain defined by $\mathbb{X}_{N}(t)=\Psi_{N}(\bs{r}_{N}(t))$, and
denote by $\mathbf{P}_{\beta,k}^{(1)}$, $0\le k\le2$, the law of a
$\{0,1,2\}$-valued Markov chain which starts from $k$ and which jumps
from $i$ to $j$ at rate $r(i,j)=\omega_{\beta}/\nu_{\beta}$.  Next
theorem follows from \cite[Theorem 2.1]{LS} and from assertions (1)
and (2) of Proposition \ref{prow}.

\begin{theorem}
\label{thm1}
Fix $\beta\in(\beta_{1},\infty)$, $0\le i\le2$, and a sequence
$\{\bs{x}_{N} : N\ge 1\}$ such that
$\bs{x}_{N}\in\mc{E}_{\beta}^{N}(i)$ for all $N$. Then, under
$\mathbb{P}_{\bs{x}_{N}}^{N}$, the law of the rescaled hidden Markov
chain $\mathbb{X}_{N}(2\pi Ne^{\theta_{\beta}N}t)$ converges to
$\mathbf{P}_{\beta,i}^{(1)}$ in the soft topology \cite{Lan1}.
\end{theorem}

It is notable that the limiting Markov chain is reversible, while the
underlying dynamic is non-reversible. 

We may interpret Theorem \ref{thm1} in a more intuitive form. Consider
the process $\bs{r}_{N}(t)$ starting from a point $\bs{x}_{N}$ in the
valley $\mc{E}_{\beta}^{N}(i)$, and denote by
$H_{\mc{E}_{\beta}^{N}\setminus\mc{E}_{\beta}^{N}(i)}$ the hitting
time of one of the other valleys. Theorem \ref{thm1} asserts that
\begin{equation*}
\mathbb{E}^N_{\bs{x}_{N}}
\big[H_{\mc{E}_{\beta}^{N}\setminus\mc{E}_{\beta}^{N}(i)}\big]
\;=\;\left[1+o_{N}(1)\right]\,\frac{1}{2}\,
\frac{\nu_{\beta}}{\omega_{\beta}}\,
2\pi Ne^{\theta_{\beta}N} \;=\; 
\left[1+o_{N}(1)\right]\,\frac{\nu_{\beta}}{\omega_{\beta}}
\pi Ne^{\theta_{\beta}N}\;,
\end{equation*}
that under $\bb P^N_{\bs{x}_{N}}$
\begin{equation*}
\frac{\omega_{\beta}}{\nu_{\beta}} \, \frac 1{\pi N e^{\theta_{\beta}N}}\,
H_{\mc{E}_{\beta}^{N}\setminus\mc{E}_{\beta}^{N}(i)} \;\text{
  converges to a mean-one exponential random variable}\;,
\end{equation*}
and that $\bs{r}_{N}(t)$ jumps to one of the other two valleys with
asymptotically equal probability.

\smallskip\noindent\textbf{Regime II:
  $\beta\in(\beta_{2},\beta_{1})$.} In this range of temperatures,
there are four metastable sets and two time-scales. In the time scale
$2\pi N \exp\{\theta_{\beta}(3)N\}$, starting from a point in the
valley $\mc{E}_{\beta}^{N}(3)$, after an exponential time the process
jumps to one of the other three valleys $\mc{E}_{\beta}^{N}(i)$, $0\le
i\le2$, and remains there for ever (in this time scale).  In the
longer time scale $2\pi N \exp\{\theta_{\beta}N\}$ the process
exhibits a tunneling behavior between the metastable sets
$\mc{E}_{\beta}^{N}(i)$, $0\le i\le2$.

A rigorous statement requires some notations. Define the projection
map 
\begin{equation*}
\widehat{\Psi}_{N}(\bs{x})\;=\;\sum_{i=0}^{3}i\,
\mathbf{1}\{\bs{x}\in\mc{E}_{\beta}^{N}(i)\}
\;+\; N\,\mathbf{1}\big\{ \bs{x}\in \widehat{\Delta}_N \big\} \;,
\end{equation*}
where $\widehat{\Delta}_N = \Xi_{N}\setminus\cup_{0\le i\le 3}\,
\mc{E}_{\beta}^{N}(i)$.  Let $\widehat{\mathbb{X}}_{N}(t) =
\widehat{\Psi}_{N} (\bs{r}_{N}(t))$, and recall that we represent by
$\mathbb{X}_{N}(t)$ the process $\Psi_{N} (\bs{r}_{N}(t))$.  Denote by
$\mathbf{P}_{\beta,k}^{(2)}$ the law of the $\{0,1,2\}$-valued Markov
chain which starts from $k$ and whose jump rates are given by
$r(i,j)=\omega_{\beta}/(3\nu_{\beta})$, $0\le i\neq j\le2$. Similarly,
denote by $\mathbf{Q}_{\beta,k}^{(2)}$ the law of the
$\{0,1,2,3\}$-valued Markov chain which starts from $k$ and whose jump
rates are given by
\begin{equation*}
r(i,j)\;=\;\mathbf{1}\{i=3\} \,
\frac{\omega_{\beta}}{\nu_{\beta}(3)}\;, \quad 0\le i\neq j\le3\;.
\end{equation*}
Note that the points $0$, $1$, $2$ are absorbing for the chain
$\mathbf{Q}_{\beta,k}^{(2)}$.

Next theorem follows from \cite[Theorem 2.1]{LS}, from \cite[displays
(2.12), (2.13)]{LS}, and from assertions (4) and (5) of Proposition
\ref{prow}.

\begin{theorem}
\label{thm2} 
Fix $\beta\in(\beta_{2},\beta_{1})$, $0\le i\le3$, $0\le j\le 2$ and
sequences $\{\bs{x}_{N} : N\ge 1\}$, $\{\bs{y}_{N} : N\ge 1\}$ such
that $\bs{x}_{N}\in\mc{E}_{\beta}^{N}(i)$,
$\bs{y}_{N}\in\mc{E}_{\beta}^{N}(j)$ for all $N$. Then, the law of
rescaled process $\mathbb{\widehat{X}}_{N}(2\pi
Ne^{\theta_{\beta}(3)N}t)$ under $\mathbb{P}_{\bs{x}_{N}}^{N}$
converges to $\mathbf{Q}_{\beta,i}^{(2)}$ in the soft topology, and
the law of rescaled process $\mathbb{X}_{N}(2\pi
Ne^{\theta_{\beta}N}t)$ under $\mathbb{P}_{\bs{y}_{N}}^{N}$ converges
to $\mathbf{P}_{\beta,j}^{(2)}$ in the soft topology.
\end{theorem}

Therefore, in the time scale $2\pi N \exp\{\theta_{\beta}(3)N\}$,
starting from a point in $\mc{E}_{\beta}^{N}(3)$, after a mean
$\nu_{\beta}(3)/3\omega_{\beta}$ exponential time, the chain jumps to
one of the deeper valleys $\mc{E}_{\beta}^{N}(i)$, $0\le i\le2$, with
equal probability. After this jump, in the time scale $2\pi N
\exp\{\theta_{\beta}(3)N\}$, the chain is trapped in the deeper
valley reached.

In the time scale $2\pi N \exp\{\theta_{\beta}N\}$, the process
exhibits a tunneling behavior, similar to the one observed in regime
I, with the notable difference that the jump rate is dropped by a
factor $3$. 

The discontinuity of the jump rate is due to the change of the
inter-valley structure. While in regime I, the valleys
$\mc{E}_{\beta}^{N}(0)$ and $\mc{E}_{\beta}^{N}(1)$ are connected
directly by the saddle point $\bs{\sigma}_{\beta}^{2}$, in regime II
these two valleys are indirectly connected via the shallower valley
$\mc{E}_{\beta}^{N}(3)$, the process has to overcome the two saddle
points $\bs{\sigma}_{\beta}^{0}$ and $\bs{\sigma}_{\beta}^{1}$ in
order to make transition from $\mc{E}_{\beta}^{N}(0)$ to
$\mc{E}_{\beta}^{N}(1)$. After reaching the well
$\mc{E}_{\beta}^{N}(3)$, the chain may return to
$\mc{E}_{\beta}^{N}(0)$ before reaching $\mc{E}_{\beta}^{N}(1)$,
slowing down the transition rate between the valleys.

\smallskip\noindent\textbf{Regime III:
  $\beta\in(\beta_{3},\beta_{2})$.} In this regime, there are three
metastable sets and one stable set. In the time scale $2\pi N
\exp\{\theta_{\beta}N\}$, starting from a point in one of the valleys
$\mc{E}_{\beta}^{N}(i)$, $0\le i\le 2$, after an exponential time the
process jumps to the well $\mc{E}_{\beta}^{N}(3)$ and stays there for
ever.

To describe this metastable behavior more precisely, denote by
$\mathbf{Q}_{\beta,k}^{(3)}$ the $\{0,1,2,$ $3\}$-valued Markov chain
starting from $k$ whose jump rates are given by
\begin{equation*}
r(i,j)\;=\;\mathbf{1}\{j=3\} \,
\frac{\omega_{\beta}}{\nu_{\beta}}\;, \quad 0\le i\neq j\le3\;.
\end{equation*}
Note that the point $3$ is an absorbing state for this chain.

Recall the definition of the process $\mathbb{\widehat{X}}_{N}(t)$
introduced in the previous regime. 

\begin{theorem}
\label{thm3}
Fix $\beta\in(\beta_{3},\beta_{2})$, $0\le i\le3$, and a sequence
$\{\bs{x}_{N} : N\ge 1\}$ such that
$\bs{x}_{N}\in\mc{E}_{\beta}^{N}(i)$ for all $N$. Then, under
$\mathbb{P}_{\bs{x}_{N}}^{N}$, the law of rescaled process
$\mathbb{\widehat{X}}_{N}(2\pi Ne^{\theta_{\beta}N}t)$ converges to
$\mathbf{Q}_{\beta,i}^{(3)}$ in the soft topology.
\end{theorem}

\smallskip\noindent\textbf{Dynamics at the critical temperatures.}
For $\beta=\beta_{3}$, the point $\bs{p}$ is the unique minima, which
is the global minima, and thus no metastability or tunneling
phenomenon occurs.

For $\beta=\beta_{2}$, the four metastable sets $\mc{E}_{N}(i)$,
$0\le i\le3$, have the same depth, i.e., $\theta_{\beta}=\theta_{\beta}(3)$,
and the process exhibits the tunneling behavior among them. 

Denote by $\mathbf{P}_{\beta,k}^{(4)}$ the $\{0,\,1,\,2,\,3\}$-valued
Markov chain starting from $k$ whose jump rates are given by
\begin{equation*}
r(i,j)\;=\;\mathbf{1}\{ j=3\}\,
\frac{\omega_{\beta}}{\nu_{\beta}} \;+\;
\mathbf{1}\{i=3\} \, \frac{\omega_{\beta}}{\nu_{\beta}(3)}\;, \quad
0\le i\neq j\le3\;,
\end{equation*}
and recall the definition of the process
$\mathbb{\widehat{X}}_{N}(t)$.

\begin{theorem}
\label{thm4}
Fix $\beta=\beta_{2}$, $0\le i\le 3$, and a sequence $\{\bs{x}_{N} :
N\ge 1\}$ such that $\bs{x}_{N}\in\mc{E}_{\beta}^{N}(i)$ for all
$N$. Then, under $\mathbb{P}_{\bs{x}_{N}}^{N}$, the law of rescaled
process $\mathbb{\widehat{X}}_{N}(2\pi Ne^{\theta_{\beta}N}t)$
converges to $\mathbf{P}_{\beta,i}^{(4)}$ in the soft topology.
\end{theorem}

For $\beta=\beta_{1}$, the metastable behavior cannot be obtained by
the approach presented in \cite{LS}. We can expect that the process
exhibits a tunneling behavior among the three metastable valleys
$\mc{E}_{\beta}^{N}(i)$, $0\le i\le 2$. In view of assertion (3) of
Proposition \ref{prow}, the transitions may occur by crossing the
point $\bs{p}$. But, $\bs{p}$ is not a saddle point. Instead, the
Hessian at $\bs{p}$ is the zero matrix, and therefore the potential is
flat around $\bs{p}$. In consequence, we expect that the process
behaves like a diffusion around $\bs{p}$.  However, the precise jump
rates cannot be computed by the method of \cite{LS}, and the
derivation of the metastable behavior of this chain for
$\beta=\beta_1$ requires new ideas.

\section{Non-zero External Magnetic Field}
\label{sec5}

We examine in this section the metastable behavior of the mean-field
Potts model with an external field.

In Subsection \ref{sec51}, with perturbative arguments, we extend the
results of the previous section to the case in which the external
field is small. Even though the assertions are not stated as theorems,
all results presented in this subsection are rigorous and can be
formulated as the ones in the previous section.

In Subsections \ref{sec52} and \ref{sec53}, we present the metastable
behavior of the Potts model in the cases where $\beta>2$,
$\theta_\text{e}=(2k+1)\pi/3$ and $\theta_\text{e}=2k\pi/3$,
respectively. For large enough $r$, as the external field tilts the
potential significantly, it is not difficult to guess the metastable
behavior of the system. The interesting question is the existence of
intermediate regimes between small and large external fields. In the
case $\theta_\text{e}=(2k+1)\pi/3$, $k\in\mathbb{Z}$, there are
indeed two critical strengths of the external field, $0<r^\beta_1<
r^\beta_2<\infty$, and a new regime appears for $r\in
(r_1^\beta,\,r_2^\beta)$. In contrast, for $\theta_\text{e}=2k\pi/3$,
$k\in\mathbb{Z}$, there is only one critical strength of the external
field and we do not observe intermediate regimes.

In the case $\theta_\text{e} \neq k\pi/3$, $k\in\mathbb{Z}$, although
we can derive the metastable behavior of the system by numerical
computations, it seems impossible to obtain rigorous results in a
concrete form. The reason for the lack of rigorous results in the
general case is that the method to derive the metastable behavior
requires the identification of the critical points of the potential,
and the computation of the eigenvalues of the Hessian of the potential
at the critical points. The equations for the critical points, which
in the case of zero external field corresponds to the identities
\eqref{cr1}, can not be solved explicitly in the case of a positive
external field, at least in general. This case is thus left to the
realm of numerical computations.

\subsection{Small external fields}
\label{sec51}

The case of a small external field can be examined by perturbative
arguments. The external field may break the spin symmetry by favoring
one or two values. There are many different possible regimes depending
on the orientation of the external magnetic field and on the value of
the temperature. We examine in this subsection three cases at low
temperature to eliminate the entropic set. A similar analysis can be
carried out for temperatures lying in the intervals
$(\beta_{2},\beta_{1})$ and $(\beta_{3},\beta_{2})$.

Assume that $\beta>\beta_1$. By symmetry, there are three cases to be
considered: the case in which the external field is aligned with one
spin, creating one stable set and two symmetric metastable sets, the
case in which the external field takes the mean value between two
spins, and the case in which it takes any other value.

The structure of the potential $F_{\beta} $, presented in Proposition
\ref{pro1}, is not perturbed significantly if the external field is
small. Fix an angle $\theta_{\text{e}}$ and regard the potential
$F_{\beta}$ as a function of $\bs{x}=(x_{1},\,x_{2})$ and
$r=r_{\text{e}}$:
\begin{equation}
\label{F_b}
F_{\beta}(\bs{x},r)\;=\;F_{\beta}(\bs{x}) \;-\;
r\sum_{i=0}^{2} x_{i} \, \cos \left(\theta_{\text{e }}-
\frac{2\pi i}{3}\right)\;,
\end{equation}
where $F_{\beta}$ is the potential introduced in (\ref{pot2}).  Let
$K_{\beta}:\Xi\times[0,\infty)\rightarrow\mathbb{R}^{2}$ be given by
\begin{equation*}
K_{\beta}(\bs{x},r)\;=\; \big(\partial_{x_{1}}F_{\beta}(\bs{x},r),
\,\partial_{x_{2}}F_{\beta}(\bs{x},r)\big)\;,
\end{equation*}
where $\partial_{x_i}F_{\beta}$ represents the partial derivative
of $F_{\beta}$ with respect to $x_i$.

Recall the definition of the point $\bs{m}_{0}^{\beta}$ introduced in
\eqref{01}. By Proposition \ref{pro1} and by its proof,
$K(\bs{m}_{0}^{\beta},0)=0$ and the Jacobian of $K_{\beta}(\cdot,0)$
at $\bs{m}_{0}^{\beta}$, which is the Hessian of $F_{\beta}(\cdot,0)$
at $\bs{m}_{0}^{\beta}$, is non-degenerate.  Hence, by the implicit
function theorem, there exist $\epsilon = \epsilon(\beta)>0$ and a
smooth function $\bs{m}_{0}^{\beta}(\cdot): [0,\epsilon) \rightarrow
\mathbb{R}^{2}$ such that
\begin{equation*}
\bs{m}_{0}^{\beta}(0)=\bs{m}_{0}^{\beta}\;\;\;\mbox{and}\;\;\;
K_{\beta}\big(\bs{m}_{0}^{\beta}(r),r\big)\;=\;0\;,\;\;
\forall \, r\in[0,\epsilon)\;.
\end{equation*}
In other words, $\bs{m}_{0}^{\beta}(r)$ is a critical point of
$F_{\beta}(\cdot,r)$. The same argument can be applied to the other
critical points $\bs{p}$, $\bs{m}_{i}^{\beta}$,
$\bs{\sigma}_{i}^{\beta}$, $0\le i\le 2$. Moreover, no new critical
points appear for small enough $r_{\text{e}}$.

Let $\phi_{i}^{\beta}(r)=F_{\beta}(\bs{m}_{i}^{\beta}(r),r)$ and
$\psi_{i}^{\beta}(r)=F_{\beta}(\bs{\sigma}_{i}^{\beta}(r),r)$, and
recall the definition of the points $p_\beta$, $q_\beta$ introduced
just above \eqref{01}.

\begin{lemma}
\label{lem4}
For $0\le i\le2$,
\begin{equation*}
\frac{\textup{d}\phi_{i}^{\beta}}{\textup{d}r}(0)
\;=\;(3p_{\beta}-1) \cos \Big(\theta_{\textup{e}} 
-\frac{2\pi i}{3} \Big)\;, \quad
\frac{\textup{d}\psi_{i}^{\beta}}{\textup{d}r}(0)
\;=\;(3q_{\beta}-1)\cos \Big(\theta_{\textup{e}}
- \frac{2\pi i}{3} \Big)\;,
\end{equation*}
\end{lemma}

\begin{proof}
We present the computations for $i=0$, the other ones being analogous.
By the definition of $\bs{m}_{0}^{\beta}$, by \eqref{F_b}, and by the
chain rule,
\begin{align*}
\frac{\textup{d}\phi_{0}^{\beta}}{\textup{d}r}(0) \; &=\; 
(\nabla_{r}\bs{m}_{0}^{\beta})(0) \cdot 
K_{\beta}\big(\bs{m}_{0}^{\beta}(0),0\big) 
\;-\; (1-2p_{\beta})\, \cos\theta_{\text{e }} \\
\;&-\;  p_{\beta}\, \cos\Big(\theta_{\text{e }}-\frac{2\pi}{3}\Big)
\;-\; p_{\beta}\cos\Big(\theta_{\text{e }}-\frac{4\pi}{3}\Big)\;.
\end{align*}
By definition of $\bs{m}_{0}^{\beta} (0)$, the first term vanishes.
The other terms can be computed to provide the first identity of the
lemma.  The calculations for the second identity are similar.
\end{proof}

We are now in a position to present the metastable behavior of the
magnetization under a small external magnetic field. Recall from the
previous section that for $\beta>\beta_1$ and $r_{\text{e}}=0$,
$(1/3,1/3)$ is a local maximum, the points $\bs{m}_{i}^{\beta}$ are
local minima, and the points $\bs{\sigma}_{i}^{\beta}$ are saddle
points.  All local minima are at the same height, as well as all
saddle points.

\smallskip
\noindent\textbf{Case I: $\theta_{\text{e}}=2k\pi/3$, $k=0,\,1,\,2$.}
To fix ideas, suppose that $k=0$.  By Lemma \ref{lem4}, and since
$F_{\beta}(\bs{m}_{0}^{\beta}) = F_{\beta}(\bs{m}_{k}^{\beta})$,
$F_{\beta}(\bs{\sigma}_{0}^{\beta}) =
F_{\beta}(\bs{\sigma}_{k}^{\beta})$, $k=1$, $2$, and $p_{\beta}< 1/3
<q_{\beta}$, there exists $\epsilon (\beta)>0$ such that for all $r_{\text{e}}< \epsilon (\beta)$, 
\begin{equation}
\label{m11}
\begin{aligned}
& F_{\beta}(\bs{m}_{0}^{\beta}(r_{\text{e}}),r_{\text{e}})
\;<\;F_{\beta}(\bs{m}_{1}^{\beta}(r_{\text{e}}),r_{\text{e}})
\;=\;F_{\beta}(\bs{m}_{2}^{\beta}(r_{\text{e}}),r_{\text{e}})\;,\\
&\quad F_{\beta}(\bs{\sigma}_{0}^{\beta}(r_{\text{e}}),r_{\text{e}})
\;>\;F_{\beta}(\bs{\sigma}_{1}^{\beta}(r_{\text{e}}),r_{\text{e}})
\;=\;F_{\beta}(\bs{\sigma}_{2}^{\beta}(r_{\text{e}}),r_{\text{e}})\;.
\end{aligned}
\end{equation}

By (\ref{m11}), for $0 <r_{\text{e}}< \epsilon (\beta)$,
$\bs{m}_{1}^{\beta}(r_{\text{e}})$ and
$\bs{m}_{2}^{\beta}(r_{\text{e}})$ are the bottom points of metastable
sets with the same height, and $\bs{m}_{0}^{\beta}(r_{\text{e}})$,
being the global minima, is the bottom point of a stable set. By
(\ref{m11}), in an appropriate time scale, starting in a neighborhood
of $\bs{m}_{1}^{\beta} (r_{\text{e}})$, after an exponential time, the
chain $\bs r_N(t)$ jumps to $\bs{m}_{0}^{\beta} (r_{\text{e}})$ by
crossing the saddle point $\bs{\sigma}_{2}^{\beta}(r_{\text{e}})$. The
expectation of the transition time is given by
\begin{equation}
\label{et1}
\left[1+o_{N}(1)\right]\,2\pi N\,
\frac{\nu(\bs{m}_{1}^{\beta}(r_{\text{e}}))}
{\omega(\bs{\sigma}_{2}^{\beta}(r_{\text{e}}))}
\,\exp\left\{ N[F_{\beta}(\bs{\sigma}_{2}^{\beta}
(r_{\text{e}}),r_{\text{e}})-F_{\beta}(\bs{m}_{1}^{\beta}
(r_{\text{e}}),r_{\text{e}})]\right\} \;,
\end{equation}
where $\nu(\bs{m}_{1}^{\beta}(r_{\text{e}}))$ and
$\omega(\bs{\sigma}_{2}^{\beta}(r_{\text{e}}))$ are defined as in
(\ref{nu2}) and (\ref{omega}), respectively.  

The metastable behavior of this model is thus described by a $3$-state
Markov chain with one absorbing point and two other points which may
jump only to the absorbing point.

\smallskip
\noindent\textbf{Case II: $\theta_{\text{e}}=(2k+1)\pi/3$,
  $k=0,\,1,\,2$}. 
Suppose, without loss of generality, that $k=0$. Then, by the same argument as Case I, we have that
\begin{gather*}
F_{\beta}(\bs{m}_{0}^{\beta}
(r_{\text{e}}),r_{\text{e}})\;=\;
F_{\beta}(\bs{m}_{1}^{\beta}(r_{\text{e}}),r_{\text{e}})\;<\;
F_{\beta}(\bs{m}_{2}^{\beta}(r_{\text{e}}),r_{\text{e}})\;,\\
F_{\beta}(\bs{\sigma}_{0}^{\beta}
(r_{\text{e}}),r_{\text{e}})\;=\;
F_{\beta}(\bs{\sigma}_{1}^{\beta}(r_{\text{e}}),r_{\text{e}})
\;>\;F_{\beta}(\bs{\sigma}_{2}^{\beta}(r_{\text{e}}),r_{\text{e}})\;,
\end{gather*}
for sufficiently small $r$. Hence, there are two different metastable behaviors associated to two
different heights: $h^\beta_{01} = F_{\beta}(\bs{\sigma}_{0}^{\beta}
(r_{\text{e}}), r_{\text{e}})$ and $h^\beta_2 = F_{\beta}
(\bs{\sigma}_{2}^{\beta} (r_{\text{e}}),r_{\text{e}})$, with
$h^\beta_2 < h^\beta_{01}$.

The height $h^\beta_{01}$ defines two valleys. More precisely the set
$\{\bs x \in \Xi : F_{\beta}(\bs x , r_{\text{e}}) \le h^\beta_{01}\}$ can
be written as $\overline{V_2} \cup \overline{V_{01}}$, where the open
set $V_2$ contains the point $\bs{m}_{2}^{\beta}(r_{\text{e}})$, the
open set $V_{01}$ contains the points $\bs{m}_{0}^{\beta}
(r_{\text{e}})$, $\bs{m}_{1}^{\beta} (r_{\text{e}})$, and
$\overline{V_2} \cap \overline{V_{01}} =
\{\bs{\sigma}_{0}^{\beta}(r_{\text{e}}),
\bs{\sigma}_{1}^{\beta}(r_{\text{e}})\}$.

In a certain time scale, related to the difference
$F_{\beta}(\bs{\sigma}_{0}^{\beta} (r_{\text{e}}),r_{\text{e}}) -
F_{\beta}(\bs{m}_{2}^{\beta} (r_{\text{e}}),r_{\text{e}})$, starting
from a neighborhood of $\bs{m}_{2}^{\beta}(r_{\text{e}})$, after an
exponential time, the process jumps to one of the two stable sets. The
expectation of the transition time can be computed as in Case I.

The height $h^\beta_{2}$ defines also two valleys: the set $\{\bs x
\in \Xi : F_{\beta}(\bs x , r_{\text{e}}) \le h^\beta_2\}$ can be
written as $\overline{W_0} \cup \overline{W_{1}}$, where the open set
$W_0$ contains the point $\bs{m}_{0}^{\beta}(r_{\text{e}})$, the open
set $W_{1}$ contains the point $\bs{m}_{1}^{\beta} (r_{\text{e}})$,
and $\overline{W_0} \cap \overline{W_{1}} =
\{\bs{\sigma}_{2}^{\beta}(r_{\text{e}})\}.$

In a time scale related to the difference
$F_{\beta}(\bs{\sigma}_{2}^{\beta} (r_{\text{e}}),r_{\text{e}}) -
F_{\beta}(\bs{m}_{0}^{\beta} (r_{\text{e}}),r_{\text{e}})$, the
process jumps at exponential times from a neighborhood of
$\bs{m}_{0}^{\beta}(r_{\text{e}})$ to a neighborhood of
$\bs{m}_{1}^{\beta}(r_{\text{e}})$, and reciprocally. Here also the
expectation of the transition time can be computed.

\smallskip
\noindent\textbf{Case III: $\theta_{\text{e}}\neq k\pi/3$ for all
$k\in\mathbb{Z}$.} To fix ideas, suppose without loss of generality
that $0<\theta_{\text{e}}<\frac{\pi}{3}$.  By Lemma \ref{lem4}, 
\begin{gather*}
F_{\beta}(\bs{m}_{0}^{\beta}(r_{\text{e}}),r_{\text{e}})
\;<\;F_{\beta}(\bs{m}_{1}^{\beta}(r_{\text{e}}),r_{\text{e}})
\;<\;F_{\beta}(\bs{m}_{2}^{\beta}(r_{\text{e}}),r_{\text{e}})\;,\\
F_{\beta}(\bs{\sigma}_{0}^{\beta}(r_{\text{e}}),
r_{\text{e}})\;>\;F_{\beta}(\bs{\sigma}_{1}^{\beta}(r_{\text{e}}),
r_{\text{e}})\;>\;F_{\beta}(\bs{\sigma}_{2}^{\beta}(r_{\text{e}}),r_{\text{e}})\;.
\end{gather*}
As in Case II, there are two time scales, which might be of the same
order or even equal. In a time scale associated to the difference
$F_{\beta}(\bs{\sigma}_{1}^{\beta}(r_{\text{e}}), r_{\text{e}}) -
F_{\beta}(\bs{m}_{2}^{\beta}(r_{\text{e}}),r_{\text{e}})$, starting
from a neighborhood of $\bs{m}_{2}^{\beta}(r_{\text{e}})$, after an
exponential time, the process jumps to a neighborhood of
$\bs{m}_{0}^{\beta}(r_{\text{e}})$ and there remains for ever.

Similarly, in a time scale associated to the difference
$F_{\beta}(\bs{\sigma}_{2}^{\beta}(r_{\text{e}}), r_{\text{e}}) -
F_{\beta}(\bs{m}_{1}^{\beta}(r_{\text{e}}),r_{\text{e}})$, starting
from a neighborhood of $\bs{m}_{1}^{\beta}(r_{\text{e}})$, after an
exponential time, the process jumps to a neighborhood of
$\bs{m}_{0}^{\beta}(r_{\text{e}})$ and there remains for ever, the
neighborhood of $\bs{m}_{0}^{\beta}(r_{\text{e}})$ being a stable set.

\begin{figure}
  \protect
\includegraphics[scale=0.105]{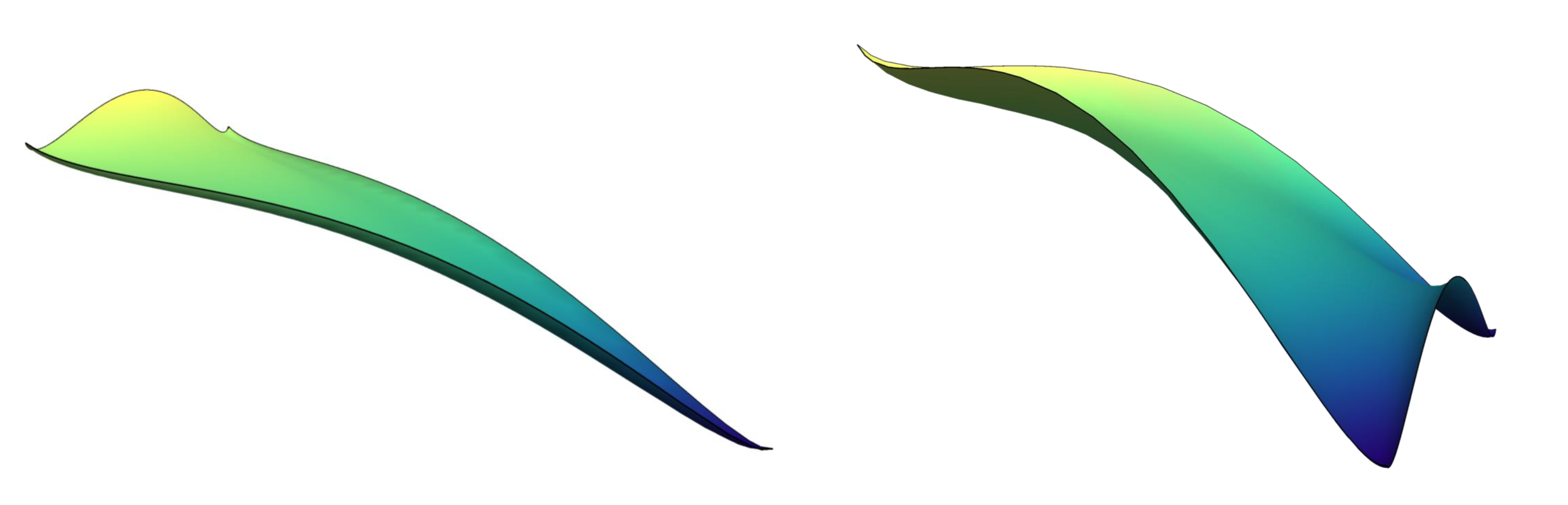}\protect
\caption{\label{fig7}The graphs of $F_\beta (\bs{x},\,r_\text{e})$ for
  $(\beta,\,r_\text{e},\,\theta_\text{e})=(2.4,0.8,0)$ (left) and
  $(\beta,r_\text{e},\,\theta_\text{e})=(2.4,0.8,\pi/3)$ (right).}
\end{figure}


\subsection{General external field with $\theta_\text{e}=(2k+1)\pi/3$}
\label{sec52}

We examine in this section the metastable behavior of the Potts model
with inverse temperature $\beta>2$ and external field equal to
$\theta_\text{e}=(2k+1)\pi/3$ for some $k=0,\,1,\,2$.  We prove that there are
three different metastable regimes depending on the magnitude of the
external field $r_{\text{e}}$. 

More precisely, fix $\beta>2$, and assume without loss of generality
that $\theta_{\text{e}}=\pi$, i.e., $k=1$. We prove below that there are two
critical values $0<r_{1}^{\beta}<r_{2}^{\beta}<1$, such that:
\begin{enumerate}
\item [(I)] For $r\in(0,\,r_{1}^{\beta})$, we observe the phenomenon
  already described in Section 5.1, and derived from a perturbative
  method. There are three local minima
  $\boldsymbol{m}_{i}^{\beta}(r)$, $0\le i\le2$, where
  $\boldsymbol{m}_{1}^{\beta}(r)$, $\boldsymbol{m}_{2}^{\beta}(r)$ are
  global minima, and there are three saddle points
  $\boldsymbol{\sigma}_{i}^{\beta}(r)$, $0\le i\le2$. The critical
  point $\boldsymbol{\sigma}_{i}^{\beta}(r)$, $i=1,\,2$, connects the
  metastable valley which contains $\boldsymbol{m}_{0}^{\beta}(r)$ to
  the valley which contains both of $\boldsymbol{m}_{1}^{\beta}(r)$
  and $\boldsymbol{m}_{2}^{\beta}(r)$. The saddle point
  $\boldsymbol{\sigma}_{0}^{\beta}(r)$ connects the stable sets
  associated $\boldsymbol{m}_{1}^{\beta}(r)$ and
  $\boldsymbol{m}_{2}^{\beta}(r)$. In addition there are one
  additional local maxima $\boldsymbol{p}^{\beta}(r)$.

\item [(II)] For $r\in(r_{1}^{\beta},\,r_{2}^{\beta})$, we still have
  three local minima $\boldsymbol{m}_{i}^{\beta}(r)$, $0\le i\le2$,
  but there are only two saddle points
  $\boldsymbol{\sigma}_{0}^{\beta}(r)$ and $\boldsymbol{p}^{\beta}(r)$
  such that $F_{\beta} (\boldsymbol{\sigma}_{0}^{\beta}
  (r),r)<F_{\beta} (\boldsymbol{p}^{\beta} (r),r)$.  The set
  $\big\{\boldsymbol{x}:F_{\beta}(\boldsymbol{x},r) <
  F_{\beta}(\boldsymbol{p}^{\beta} (r),r)\big\}$ has two connected
  components, one of which contains the metastable local minima
  $\boldsymbol{m}_{0}^{\beta}(r)$, while the other one contains the two
  global minima $\boldsymbol{m}_{1}^{\beta}(r)$,
  $\boldsymbol{m}_{2}^{\beta}(r)$.  The set
  $\big\{\boldsymbol{x}:F_{\beta}(\boldsymbol{x},r) <
  F_{\beta}(\boldsymbol{\sigma}_{0}^{\beta}(r), r)\big\}$ consists of
  two connected components, each one containing one of the points
  $\boldsymbol{m}_{1}^{\beta}(r)$,
  $\boldsymbol{m}_{2}^{\beta}(r)$. This regime is illustrated by
  Figure \ref{fig8}.
  
  \begin{figure}
  \protect
\includegraphics[scale=0.085]{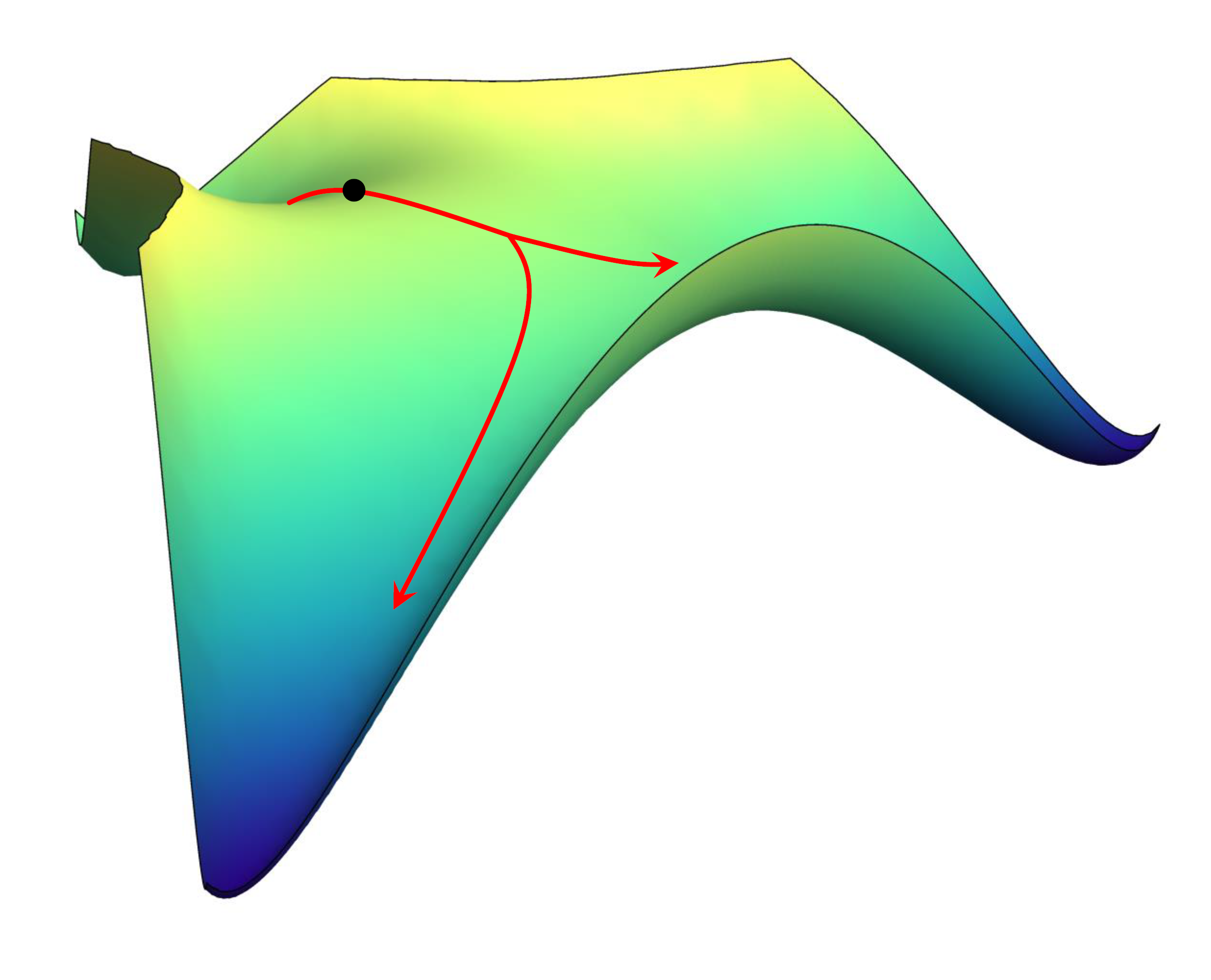}\protect
\caption{\label{fig8}The graph of $F_\beta (\bs{x},r_\text{e})$ for
  $(\beta,\,r_\text{e},\,\theta_\text{e})=(2.4,0.2,\pi/3)$.}
\end{figure}

\item [(III)] For $r\in(r_{2}^{\beta},\,\infty)$, there are three
  critical points, $\boldsymbol{m}_{1}^{\beta}(r)$,
  $\boldsymbol{m}_{2}^{\beta}(r)$ and
  $\boldsymbol{\sigma}_{0}^{\beta}(r)$. The points
  $\boldsymbol{m}_{1}^{\beta}(r)$ and $\boldsymbol{m}_{2}^{\beta}(r)$
  are global minima and the stable sets around them are connected via
  the saddle point $\boldsymbol{\sigma}_{0}^{\beta}(r)$.  This is the
  usual tunneling situation. This regime is illustrated by the right
  picture in Figure \ref{fig7}.
\end{enumerate}

\begin{remark}
  At $r=r_{1}^{\beta}$, the structure is essentially similar to case
  (II), but the critical point $\boldsymbol{p}^{\beta}(r)$ is
  degenerate, and the approach does not apply. On the other hand, at
  $r=r_2^\beta$, the situation and the result are analogous the the
  ones of case (III).
\end{remark}

\begin{figure}
\label{fig00}
\protect
\includegraphics[scale=0.22]{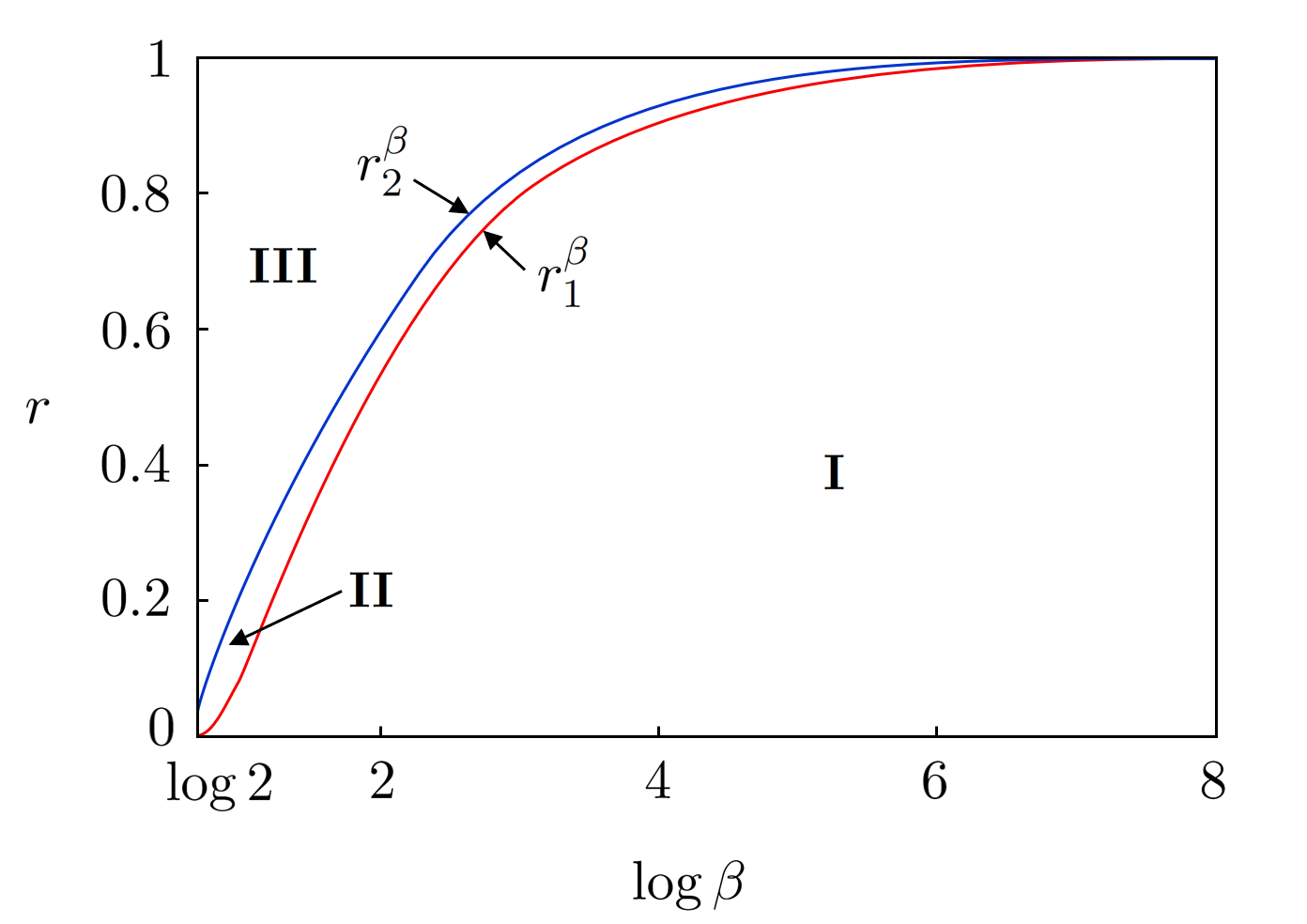}\protect
\caption{$(\beta, r)$-phase diagram at $\theta = \pi$. The regimes (I),
  (II) and (III) are indicated in the diagram.} 
\end{figure}

The critical values $r_1^\beta$ and $r_2^\beta$ have closed-form expressions given by \eqref{r1} and \eqref{r2}, respectively. 

We first characterize the critical points of $F_\beta(\cdot,r)$ for
all $r>0$. By \eqref{critical}, the critical point $(x_{1},x_{2})$
must satisfy

\begin{equation}
\label{crit3}
\frac{1}{\beta}\log x_{0}-\frac{3}{2}x_{0}+r
\;=\;\frac{1}{\beta}\log x_{1}-\frac{3}{2}x_{1}-\frac{1}{2}r
\;=\;\frac{1}{\beta}\log x_{2}-\frac{3}{2}x_{2}-\frac{1}{2}r\;.
\end{equation}

\smallskip
\noindent{\bf A. Critical points on $\{\bs{x}:x_1=x_2\}$.}
Inspired by the second equality of \eqref{crit3}, we first consider
critical points on the line $\left\{
  \boldsymbol{x}:x_{1}=x_{2}\right\}$. Denote points on this line by
$(t,t)$, $0<t<1/2$, so that $x_{0}=1-2t$. The point $(t,t)$ satisfies
\eqref{crit3} if and only if
\begin{equation*}
\frac{1}{\beta}\log(1-2t) \;-\; \frac{3}{2}(1-2t) \;+\; r\;=\;
\frac{1}{\beta}\log t \;-\; \frac{3}{2}t \;-\;
\frac{1}{2}r\;,
\end{equation*}
or equivalently $f_{r}(t)=\beta$ where 
\begin{equation*}
f_{r}(t)\;=\;\frac{2}{3(1-r-3t)}\log\frac{1-2t}{t}\;.
\end{equation*}
For $r=0$, the function $f_{0}(t)$, examined in Section 4, does not
have a singularity at $t=1/3$. In contrast, the function $f_{r}(t)$,
$0<r<1$, has a singularity at $t=(1-r)/3$. Let $k_{r}=(1-r)/3$ and
regard $f_{r}$ as a function on $(0,\,k_{r})\cup(k_{r},\,1/2)$. Next
lemma presents the elementary properties of the function $f_{r}$,
whose graph is illustrated in Figure \ref{fig9}.

\begin{lemma}
\label{alem1}
Consider the function $h:(0,\,1/2) \to \bb R$ defined by
\begin{equation*}
h(t)\;=\;-3t(1-2t)\log\frac{1-2t}{t}-3t+1\;.
\end{equation*}
Then, 
\begin{enumerate}
\item For $0<r<1$ and $t\in(0,\,k_{r})\cup(k_{r},\,1/2)$, $f_{r}'(t)$
  and $r-h(t)$ have the same sign. In particular, $f_{r}'(t)=0$ if and
  only if $h(t)=r$.

\item For any $0<r<1$, the equation $f_{r}'(t)=0$ has unique solution
  $m_{0}(r)\in(0,\,k_{r})$. The function $f_{r}(\cdot)$ is decreasing
  on $(0,\,m_{0}(r))$, and is increasing on
  $(m_{0}(r),\,k_{r})\cup(k_{r},\,1/2)$.  Furthermore,
\begin{equation*}
\lim_{t\downarrow0}f_{r}(t)=\lim_{t\uparrow k_{r}}f_{r}(t)
\;=\; \lim_{t\uparrow1/2}f_{r}(t)=\infty\;\;
\mbox{and\;\;}\lim_{t\downarrow k_{r}}f_{r}(t)=-\infty\;.
\end{equation*}

\item For all $r\ge1$, the function $f_{r}(\cdot)$ is increasing on
  $(0,\,1/2)$, and $\lim_{x\downarrow0}f_{r}(x)=-\infty$,
  $\lim_{x\uparrow1/2}f_{r}(x)=\infty$.

\item On $(0,\,1)$, the map $r\mapsto m_{0}(r)$ is decreasing and
  $\lim_{r\uparrow1}m_{0}(r)=0$

\item On $(0,\,1)$, the map $r\mapsto f_{r}(m_{0}(r))$ is increasing
  and $\lim_{r\uparrow1}f_{r}(m_{0}(r))=\infty$.
\end{enumerate}
\end{lemma}

\begin{figure}
\protect
\includegraphics[scale=0.247]{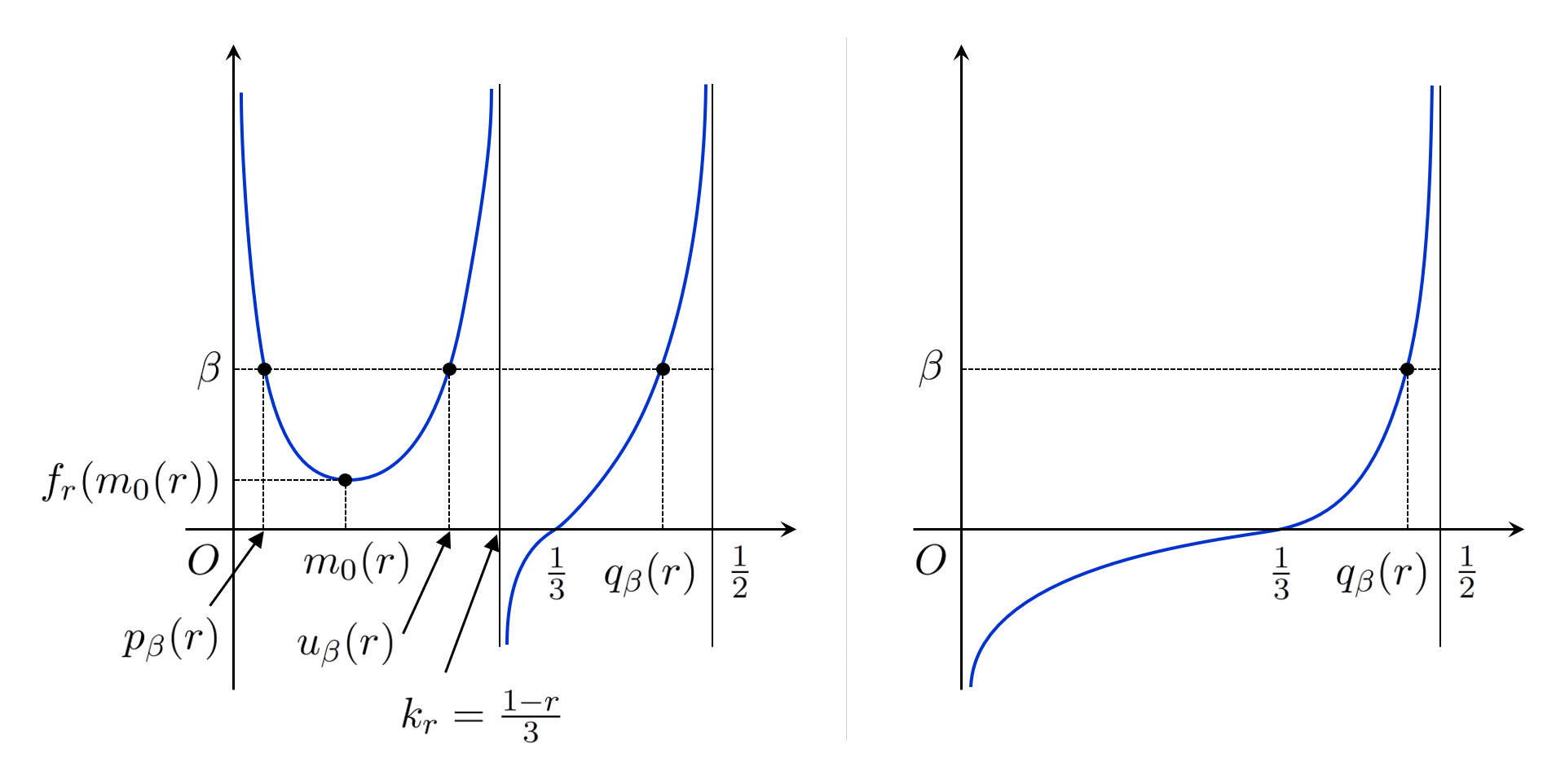}\protect
\caption{\label{fig9}The graph of $f_{r}(\cdot)$ for $r\in(0,1)$ (left)
and $r\in[1,\infty)$ (right)}
\end{figure}

\begin{proof}
It is easy to verify that 
\begin{equation*}
f_{r}'(t)\;=\;\frac{2}{3(1-r-3t)^{2}t(1-2t)}\left[r-h(t)\right]\;, 
\end{equation*}
so that (1) is obvious. 

We first investigate elementary properties of $h$. Since the derivative
of $h$ is given by $h'(t)=-3(1-4t)\log((1-2t)/t)$, the function
$h(t)$ is decreasing on $(0,\,1/4)\cup(1/3,\,1/2)$ and increasing
on $(1/4,\,1/3)$. We refer to Figure \ref{fig11} for the graph of
$h(t)$. Since $\lim_{t\downarrow0}h(t)=1$ and $h(1/3)=0$, we can
verify that the equation $h(t)=r$ has only one solution $m_0 (r)$ if $0<r<1$ and no solution if $r\ge1$. 

For (2), fix $0<r<1$, and then observe from the graph of $h$ that
$h(t)>r$ on $t\in(0,\,m_{0}(r))$ and $h(t)<r$ on $t\in(m_{0}(r),\,1/2)$.
We can check from an elementary calculation that $h(k_{r})<r$ and
hence $m_{0}(r)<k_{r}$. This completes the proof of the first part
of (2). The second part of (2) is direct from (1). The last part of
(2) follows easily from an elementary computation. 

For (3), since $h(t)<r$ for all $r\ge1$, the first part is obvious.
The remaining part is direct from the expression of $f_{r}(t)$. 

Assertion (4) is obvious from the fact that $h(m_{0}(r))=r$ for
$0<r<1$. To prove (5), note that $f_r(t)$ can be written as
\begin{equation*}
f_r(t) \;=\; \frac 29\, \frac {1-h(t)-3 t}{(1-r-3t) t (1-2t)}\;\cdot
\end{equation*}
Since $h(m_{0}(r))=r$, this equation becomes
\begin{equation}
\label{occ}
f_{r}(m_{0}(r))\;=\;\frac{2}{9m_{0}(r)(1-2m_{0}(r))}\;\cdot
\end{equation}
By the fact that $m_{0}(r)<m_{0}(0)=m_{0}<1/4$, where $m_{0}$ is
defined in Section 4, and that $m_{0}(r)$ is decreasing in $r$, we can
check that the right hand side of the previous displayed equation is
increasing in $r$. The second assertion of (5) follows from (3).
\end{proof}

\begin{figure}
 \protect
\includegraphics[scale=0.247]{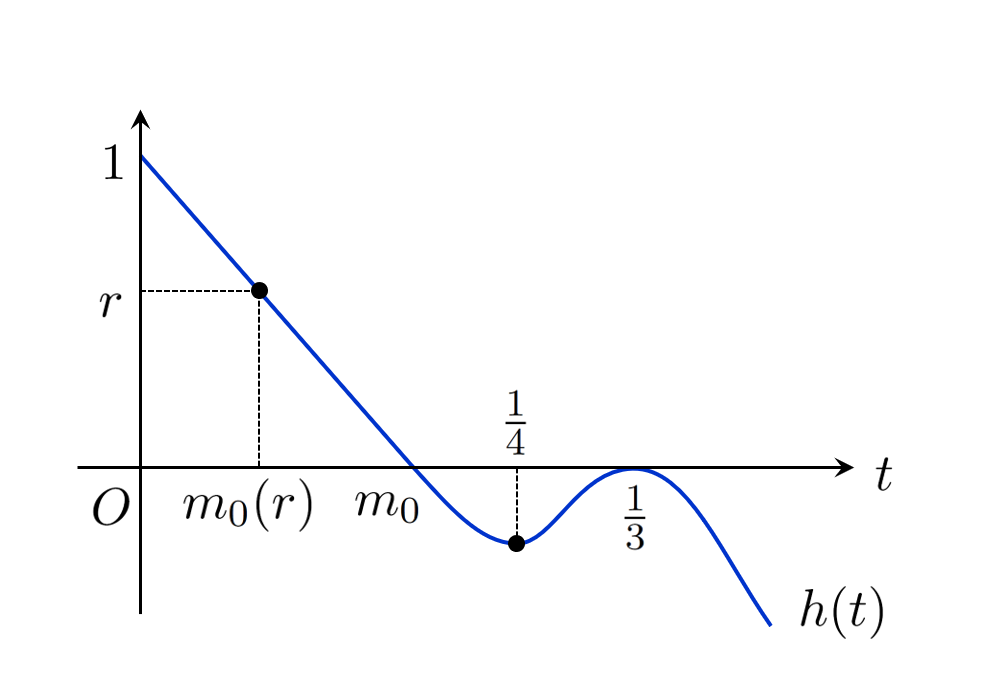}\protect
\caption{\label{fig11}The graph of $h(t)$.}
\end{figure}

Recall that $f_{0}(m_{0})=\beta_{3}<2$. Hence, by (5) of Lemma
\ref{alem1} and the intermediate value theorem, there exists unique $r\in (0,\,1)$
such that $f_{r}(m_{0}(r))=\beta$. Denote such $r$ by
$r_{2}^{\beta}$. By \eqref{occ} and the fact that $h(m_0(r))=r$, we can obtain the following formula for $r_{2}^{\beta}$:
\begin{equation}
\label{r2}
r_2^\beta\;=\;h\left(\,\frac{1}{4}-\sqrt{\frac{1}{16}-\frac{1}{9\beta}}\,\right)
\end{equation}
For $r\in(0,\,r_{2}^{\beta})$, the minimum of
$f_{r}(t)$ on $(0,\,k_{r})$, which is $f_{r}(m_{0}(r))$, is less than
$\beta$. Therefore, for such $r$, the equation $f_{r}(t)=\beta$ has
three solutions $p_{\beta}(r)<u_{\beta}(r)<q_{\beta}(r)$ on
$(0,\,1/2),$ where
\begin{equation*}
p_{\beta}\;\in\;(0,\,m_{0}(r))\;,\;\;u_{\beta}\;
\in\;(m_{0}(r),\,k_{r})\;,\;\;q_{\beta}(r)\;\in\;(1/3,\,1/2)\;.
\end{equation*}
We remark that $q_{\beta}(r)$ is larger than $1/3$ since
$f_{r}(1/3)=0$ for all $r>0$. On the other hand, for
$r\in(r_{2}^{\beta},\,\infty)$, by (3) of Lemma \ref{alem1}, there is
only one solution for $f_{r}(t)=\beta$ on $(1/3,\,1/2)$ and we denote
this again by $q_{\beta}(r)$. We refer to Figure \ref{fig9} for the
visualization.

In conclusion, there are three critical points
$\boldsymbol{m}_{0}^{\beta}(r)=(p_{\beta}(r),p_{\beta}(r))$,
$\boldsymbol{p}^{\beta}(r)=(u_{\beta}(r),u_{\beta}(r))$, and
$\boldsymbol{\sigma}_{0}^{\beta}(r)=(q_{\beta}(r),q_{\beta}(r))$ on
the line $\{\boldsymbol{x}:x_{1}=x_{2}\}$ for
$r\in(0,\,r_{2}^{\beta})$, while there is only one critical point
$\boldsymbol{\sigma}_{0}^{\beta}(r)=(q_{\beta}(r),q_{\beta}(r))$ for
$r>r_{2}^{\beta}$. Note that, for small $r$, this notation is in
accordance with the one defined in Section 5.1 for small $r$. 

In the next lemma we examine the properties of these critical
points. Let
\begin{equation}
\label{r1}
r_{1}^{\beta}\;=\; 1 \;-\; \frac{2}{\beta} \;-\;
\frac{2}{3\beta}\log\left(\frac{3\beta}{2}-2\right)\;.
\end{equation}
It is easy to check that $r_{1}^{\beta}<r_{2}^{\beta}$ for all
$\beta>2$ since $f_{r_{1}^{\beta}}(m_{0}(r_{1}^{\beta}))<\beta$. 

\begin{lemma}
\label{alem3}
We have that
\begin{enumerate}
\item The point $\boldsymbol{m}_{0}^{\beta}(r)$ is a local minimum of
  $F_{\beta}(\cdot,r)$ for all $r\in(0,\,r_{2}^{\beta})$;
\item The point $\boldsymbol{\sigma}_{0}^{\beta}(r)$ is a saddle point
  of $F_{\beta}(\cdot,r)$ for all $r>0$;
\item The point $\boldsymbol{p}^{\beta}(r)$ is a local maxima of
  $F_{\beta}(\cdot,r)$ for all $r\in(0,\,r_{1}^{\beta})$, and a saddle
  point of $F_{\beta}(\cdot,r)$ for all
  $r\in(r_{1}^{\beta},\,r_{2}^{\beta})$.
\end{enumerate}
\end{lemma}

\begin{proof}
Recall from \eqref{hess} that the determinant and the trace of the
Hessian of $F_{\beta}(\cdot,r)$ at the point $(t,\,t)$, $t\in(0,\,1/2)$,
are given by 
\begin{equation}
\label{n000}
\begin{aligned}
&( \det\, \nabla^{2}F_{\beta}) (t,t,r)\;=\;
\left(\frac{1}{\beta t}-\frac{3}{2}\right)
\left(\frac{2}{\beta(1-2t)}+\frac{1}{\beta t}-\frac{9}{2}\right)\;, \\
&\quad( \mbox{tr }\nabla^{2}F_{\beta}) (t,t,r)
\;=\;\frac{1}{\beta(1-2t)}+\frac{1}{\beta t}-3\;.
\end{aligned}
\end{equation}

To prove assertion (1), we claim that 
\begin{equation}
\label{n001}
\frac{1}{\beta p_{\beta}(r)}-\frac{3}{2}\;>\;0\;\;
\mbox{and}\;\;\frac{2}{\beta(1-2p_{\beta}(r))}
+\frac{1}{\beta p_{\beta}(r)}-\frac{9}{2}\;>\;0
\end{equation}
for all $r\in(0,\,r_{2}^{\beta}).$ For the first inequality,
substitute $\beta$ by $f_{r}(p_{\beta}(r))$ to rewrite the
inequality as
\begin{equation*}
\frac{1}{f_{r}(p_{\beta}(r))\, p_{\beta}(r)}-\frac{3}{2}\;>\;0\;.
\end{equation*}
By a straightforward computation, we can show that this inequality
is equivalent to $k(p_{\beta}(r))<1-r$ where 
\begin{equation*}
k(t)\;=\;3t \;+\; t\log\frac{1-2t}{t}\;\cdot
\end{equation*}
Since $k'(t)>0$ for $t\in(0,\,m_{0})$, we have that
$k(p_{\beta}(r))<k(m_{0}(r))$. As $m_{0}(r)$ satisfies
$f_{r}'(m_{0}(r))=0$, i.e.,
\begin{equation*}
\log\frac{1-2m_{0}(r)}{m_{0}(r)}\;=\;
\frac{1-r-3m_{0}(r)}{3m_{0}(r)(1-2m_{0}(r))}\;,
\end{equation*}
we obtain that 
\begin{equation*}
k(m_{0}(r))\;=\;3m_{0}(r)+\frac{1-r-3m_{0}(r)}{3(1-2m_{0}(r))}\;.
\end{equation*}
Since $m_{0}(r)<k_{r}=(1-r)/3$, the inequality $k(m_{0}(r))<1-r$ is
equivalent to $3(1-2m_{0}(r))>1$. Since $1-2m_{0}(r)>1-2m_{0}>1/3$, we
obtain $k(m_{0}(r))<1-r$, and thus $k(p_{\beta}(r))<1-r$. This proves
the first inequality of \eqref{n001}. For the second inequality, by
replacing $\beta$ by $f_{r}(p_{\beta}(r))$, we can reorganize the
inequality as $h(p_{\beta}(r))>r$. This is true by (1), (2) of Lemma
\ref{alem1} because $p_{\beta}(r)\in(0,\,m_{0}(r))$. This completes
the proof of \eqref{n001}.

By \eqref{n000} and \eqref{n001}, the Hessian of $F_{\beta}(\cdot,r)$
at $\boldsymbol{m}_{0}^{\beta}(r)$ is positive definite, which proves
assertion (1) of the Lemma.

We turn to (2). To prove that the Hessian of $F_{\beta}(\cdot,r)$ at
$\boldsymbol{\sigma}_{0}^{\beta} (r)$ is negative definite, it
suffices to show that
\begin{equation}
\label{n002}
\frac{1}{\beta q_{\beta}(r)}-\frac{3}{2}\;<\;0\;\;
\mbox{and}\;\;\frac{2}{\beta(1-2q_{\beta}(r))}
+\frac{1}{\beta q_{\beta}(r)}-\frac{9}{2}\;>\;0
\end{equation}
for all $r>0$. The first inequality is obvious since $\beta>2$ and
$q_{\beta}(r)>1/3$. For the second inequality, by the same type of
substitution performed in the proof of assertion (1), we can reduce
the inequality to $h(q_{\beta}(r))<r$, which follows from the proof of
Lemma \ref{alem1}, where we proved that $h(t)<0$ for $t>1/3$.

It remains to prove assertion (3). As in the first part of the proof,
we can show that
\begin{equation}
\label{n003}
\frac{2}{\beta(1-2u_{\beta}(r))} \;+\;
\frac{1}{\beta u_{\beta}(r)}-\frac{9}{2}\;<\;0
\end{equation}
for all $r\in(0,\,r_{2}^{\beta})$. 

We claim that $u_{\beta}(r)$ is decreasing in $r$. By differentiating
$f_{r}(u_{\beta}(r))=\beta$ in $r$, we obtain that
\begin{equation*}
\frac{2}{3(1-r-3u_{\beta}(r))^{2}}\left[(-1-3u_{\beta}'(r))
\log\frac{1-2u_{\beta}(r)}{u_{\beta}(r)}+\frac{u_{\beta}'(r)
(1-r-3u_{\beta}(r))}{u_{\beta}(r)(1-2u_{\beta}(r))}\right]\;=\;0\;. 
\end{equation*}
Since $f_{r}(u_{\beta}(r))=\beta$, replace
$\log\left[(1-2u_{\beta}(r))/u_{\beta}(r)\right]$ by
$\frac{3}{2}(1-r-3u_{\beta}(r))\beta$, and reorganize the
previous equality as
\begin{equation*}
u_{\beta}'(r)\left[\frac{1}{u_{\beta}(r)(1-2u_{\beta}(r))}
-\frac{9}{2}\beta\right]\;=\;\frac{3}{2}\beta\;.
\end{equation*}
By \eqref{n003}, the expression inside of the bracket is negative so
that $u_{\beta}(r)$ is decreasing in $r$.

By the definition \eqref{r1} of $r_1^\beta$ and a direct calculation, $f_{r_{1}^{\beta}}\big(2/(3\beta)\big)=\beta$.  Since
$2/(3\beta)<1/3$, $2/(3\beta)$ is either $p_{\beta}(r_{1}^{\beta})$ or
$u_{\beta}(r_{1}^{\beta})$. An elementary computation shows that
$h(2/(3\beta))<r_{1}^{\beta}$ so that
$2/(3\beta)=u_{\beta}(r_{1}^{\beta})$. Since $u_{\beta}(r)$ is
decreasing, 
\begin{equation}
\label{rcr}
\frac{1}{\beta u_{\beta}(r)}\;<\;\frac{3}{2}
\mbox{\;\;for\;\;}r\;\in\;(0,\,r_{1}^{\beta})\;\;
\mbox{and\;\;}\frac{1}{\beta u_{\beta}(r)}\;>\;
\frac{3}{2}\mbox{\;\;for\;\;}r\;\in\;(r_{1}^{\beta},\,r_{2}^{\beta})\;.
\end{equation}
At this point, by combining these results and \eqref{n003}, the proofs of the assertions (3) for
$r\in(0,\,r_{1}^{\beta})$ and $r\in(r_{1}^{\beta},\,r_{2}^{\beta})$
are essentially same to those of (1) and (2), respectively.
\end{proof}

\begin{remark}
The point $\boldsymbol{p}^{\beta}(r)$ is a degenerate critical point
of $F_{\beta}(\cdot,r)$ if $r=r_{1}^{\beta}$.
\end{remark}

\smallskip
\noindent{\bf B. Critical points not on the line $\{\bs{x}:x_1=x_2\}$.}
We now consider the critical points which are not on the line
$\{\bs{x}:x_1=x_2\}$. Let $G_{\beta}: (0,\,\infty) \to \bb R$ be the
function given by $G_{\beta}(x)= (1/\beta)\log x - (3x/2)$. With this
notation we can rewrite (\ref{crit3}) as
\begin{equation}
\label{crit4}
G_{\beta}(x_{0})+3r/2 \;=\; G_{\beta}(x_{1})=G_{\beta}(x_{2})\;.
\end{equation}
Let $l_{\beta}:=2/(3\beta)$, $g_{\beta}:=G_{\beta}(l_{\beta})$.  It is
easy to verify that $G_{\beta}$ is increasing on $(0,\,l_{\beta})$,
decreasing on $(l_{\beta},\,\infty)$, and that
$\lim_{x\downarrow0}G_{\beta}(x)=\lim_{x\uparrow\infty}G_{\beta}(x)=-\infty$.
Hence, for all $y\in(-\infty,\,g_{\beta})$, there are two solutions
$H_{\beta}(y)<l_{\beta}<K_{\beta}(y)$ of the equation
$G_{\beta}(x)=y$. Let
$H_{\beta}(g_{\beta})=K_{\beta}(g_{\beta})=l_{\beta}$, and note that
$H_{\beta}$ and $K_{\beta}$ are continuous increasing and decreasing
functions on $(-\infty,\,g_{\beta}]$, respectively. We refer to Figure
\ref{fig10}.

\begin{figure}
\protect
\includegraphics[scale=0.247]{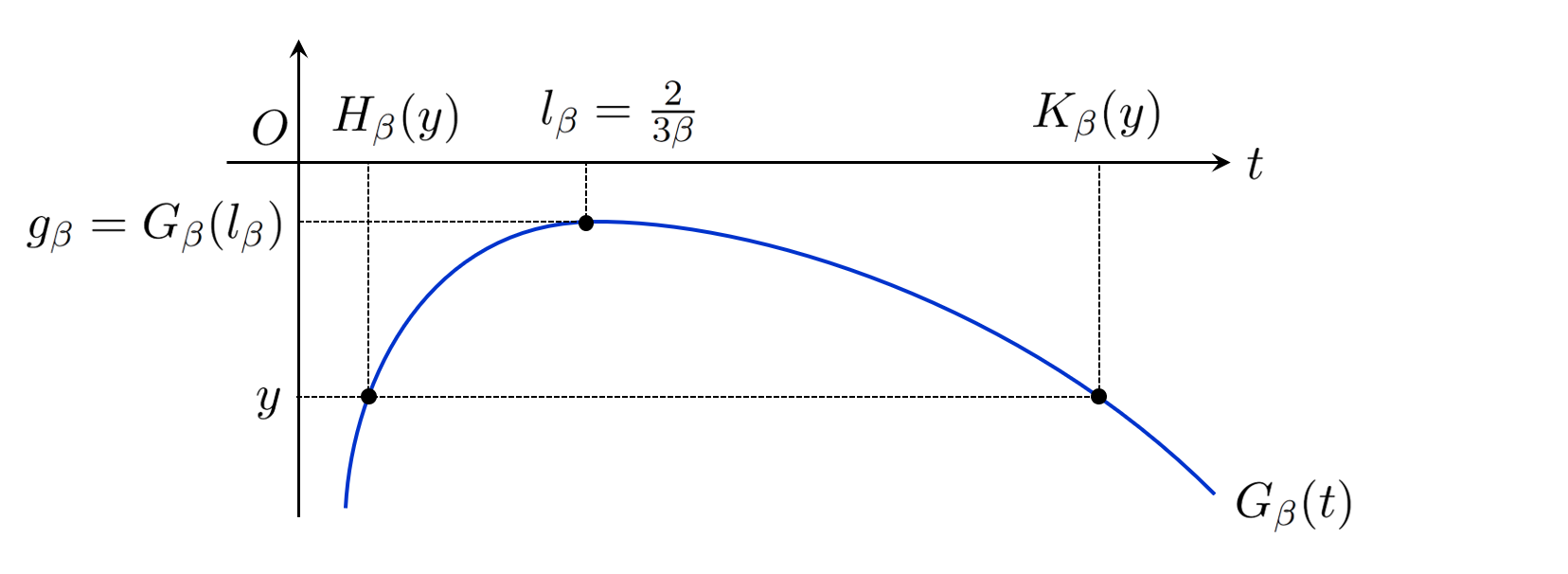}
\protect{
\caption{\label{fig10}
The graph of $G_{\beta}$ and the definitions of $H_{\beta}$,
$K_{\beta}$.}}
\end{figure}

\begin{lemma}
\label{alem4}
Let $S_{\beta}^{(k)} : (-\infty, g_\beta] \to \bb R$, $k=1,\,2$, be
the function given by $S_{\beta}^{(k)}(y) = kH_{\beta}(y) +
K_{\beta}(y)$. Then,
\begin{enumerate}
\item The function $S_{\beta}^{(1)}$ is decreasing on $(-\infty,\,g_{\beta}]$;
\item There exists $\tilde{g}_{\beta}<g_{\beta}$ such that $S_{\beta}^{(2)}$
is decreasing on $(-\infty,\,\tilde{g}_{\beta})$ and increasing on
$(\tilde{g}_{\beta},\,g_{\beta}]$.
\end{enumerate}
\end{lemma}

\begin{proof}
Since we can regard $H_{\beta}$ and $K_{\beta}$ as inverses of
$G_{\beta}$, their derivatives at $y<g_\beta$ can be written as
\begin{equation}
\label{derhk}
H_{\beta}'(y)\;=\;\beta\left[\frac{1}{H_{\beta}(y)}
\,-\, \frac{1}{l_{\beta}}\right]^{-1}\;\;
\mbox{and}\;\;\;
K_{\beta}'(y)\;=\;\beta\left[\frac{1}{K_{\beta}(y)}
\,-\,\frac{1}{l_{\beta}}\right]^{-1}\;.
\end{equation}
Since $H_{\beta}(y)<l_{\beta}<K_{\beta}(y)$ for $y<g_\beta$, the inequality
$H_{\beta}'(y)+K_{\beta}'(y)\le0$ is satisfied if
\begin{equation}
\label{HK inv}
\frac{2}{l_{\beta}}-\frac{1}{H_{\beta}(y)} \,-\,
\frac{1}{K_{\beta}(y)}\;\le\; 0\;\;
\text{or, equivalently,}
\;\left[\frac{2}{l_{\beta}}-\frac{1}{K_{\beta}(y)}\right]^{-1}
\;\ge\; H_{\beta}(y)\;.
\end{equation}
Since the left hand side of the last inequality is less than
$l_{\beta},$ since $G_{\beta}$ is increasing on $(0,l_{\beta})$, and
since $G_{\beta}(K_{\beta}(y))=G_{\beta}(H_{\beta}(y))$, this
inequality is equivalent to
$L_{\beta}^{(1)}\big(K_{\beta}(y)\big)\ge0$ where
\begin{equation*}
L_{\beta}^{(1)}(t)\;=\;G_{\beta}\left(\left[
\frac{2}{l_{\beta}}-\frac{1}{t}\right]^{-1}\right)
\,-\, G_{\beta}(t)\;,\;\;t\ge l_{\beta}\;.
\end{equation*}
It is easy to see that $L_{\beta}^{(1)}$ is increasing so that
$L_{\beta}^{(1)}(t)\ge L_{\beta}^{(1)}(l_{\beta})=0$. This proves
assertion (1) of the lemma.

The proof of assertion (2) is analogous. By the arguments presented
above, the sign of the function $2H_{\beta}'(y)+K_{\beta}'(y)$, $y\le
g_\beta$, is the opposite sign of
$L_{\beta}^{(2)}\big(K_{\beta}(y)\big)$, where
\begin{equation*}
L_{\beta}^{(2)}(t)\;=\;G_{\beta}\left(\left[\frac{3}{l_{\beta}}
\,-\, \frac{2}{t}\right]^{-1}\right)-G_{\beta}(t)\;,\;\;t
\;\ge\; l_{\beta}\;.
\end{equation*}
A direct computation shows that $L_{\beta}^{(2)}(t)$ is increasing on
$(l_{\beta},\,4l_{\beta}/3)$ and decreasing on
$(4l_{\beta}/3,\,\infty)$.  Since $L_{\beta}^{(2)}(l_{\beta})=0$ and
$\lim_{t\rightarrow\infty}L_{\beta}^{(2)}(t)=-\infty$, we can conclude
the proof of assertion (2).
\end{proof}

By \eqref{crit4}, a critical point $\boldsymbol{x}=(x_{1},x_{2})$ such
that $x_{1}\neq x_{2}$ satisfies $(x_{1},x_{2})=
\big(H_{\beta}(y),K_{\beta}(y)\big)$ or
$(x_{1},x_{2})=\big(K_{\beta}(y), H_{\beta}(y)\big)$ for some
$y<G_{\beta}(l_{\beta})$.  By symmetry, we only consider the first
case. There are two possibilities for $x_{0}$ in this situation,
namely, $x_{0}=K_{\beta}(y-3r/2)$ or $x_{0}=H_{\beta}(y-3r/2)$. We
start by considering the first case.

\begin{lemma}
\label{alem6}
The equation 
\begin{equation}
\label{6k1}
K_{\beta}(y-3r/2)+H_{\beta}(y)+K_{\beta}(y)\;=\;1
\end{equation}
has a unique solution $y_{1}^{\beta}(r)$ if $r\le r_{1}^{\beta}$ and has
no solution if $r>r_{1}^{\beta}$. Furthermore,
$\big(H_{\beta}(y_{1}^{\beta}(r)),\,K_{\beta}(y_{1}^{\beta}(r))\big)$
is a saddle point of $F_{\beta}(\cdot,r)$ if $r<r_{1}^{\beta}$.
\end{lemma}

\begin{proof}
Denote by $R_{\beta}^{(1)}(y)$ the left hand side of \eqref{6k1}.  By
(1) of Lemma \ref{alem4}, and by the fact that $K_{\beta}$ is strictly
decreasing, the function $R_{\beta}^{(1)}$ is a continuous
strictly decreasing function on $(-\infty,\,g_{\beta})$. Since
$\lim_{y\rightarrow-\infty}R_{\beta}^{(1)}(y)=\infty$, the solution of
\eqref{6k1} does uniquely exist if $R_{\beta}^{(1)}(g_{\beta})\ge1$,
and does not exist if $R_{\beta}^{(1)}(g_{\beta})<1$. Since
$H_{\beta}(g_{\beta})=K_{\beta}(g_{\beta})=l_{\beta}$, by an
elementary computation we can check that
$R_{\beta}^{(1)}(g_{\beta})\ge1$ if $r\le r_{1}^{\beta}$, and
$R_{\beta}^{(1)}(g_{\beta})<1$ if $r>r_{1}^{\beta}$.  This proves the
first part of lemma.

For the second part, it suffices to show that
$(\nabla^{2}F_{\beta})(H_{\beta}(y),K_{\beta}(y),r)$, $y\le
g_{\beta}$, has a negative determinant for all $r\le r_{1}^{\beta}$.
Let 
\begin{equation*}
a_{0}(y)\;=\;\frac{1}{\beta K_{\beta}(y-3r/2)}-
\frac{3}{2}\;,\;\;a_{1}(y)\;=\;\frac{1}{\beta H_{\beta}(y)}-
\frac{3}{2}\;,\;\;a_{2}(y)\;=\;\frac{1}{\beta K_{\beta}(y)}-
\frac{3}{2}\;\cdot
\end{equation*}
Then, by \eqref{hess}, we can write 
\begin{equation*}
\det\big[(\nabla^{2}F_{\beta}) (H_{\beta}(y),
K_{\beta}(y),r)\big]\;=\;
a_{0}(y)a_{1}(y)a_{2}(y)\left[\frac{1}{a_{0}(y)}+
\frac{1}{a_{1}(y)}+\frac{1}{a_{2}(y)}\right]\;.
\end{equation*}
Observe that $a_{0}(y),\,a_{2}(y)<0$, $a_{1}(y)>0$ and, by
\eqref{derhk},
\begin{equation*}
\frac{1}{a_{0}(y)}+\frac{1}{a_{1}(y)}+\frac{1}{a_{2}(y)} \;=\;
\frac{dR_{\beta}^{(1)}}{dy}(y)\;<\;0\;.
\end{equation*}
These observations prove that the right hand side of the penultimate
displayed equation is negative, as claimed.
\end{proof}

We turn to the second case. 

\begin{lemma}
\label{alem5}
For all $r>0$, the equation
\begin{equation}
\label{6k2}
H_{\beta}(y-3r/2)+H_{\beta}(y)+K_{\beta}(y)\;=\;1
\end{equation}
has unique solution $y_{2}^{\beta}(r)$. Furthermore, the point
$\big(H_{\beta}(y_{2}^{\beta}(r)),\,K_{\beta}(y_{2}^{\beta}(r))\big)$
is a local minimum of $F_{\beta}(\cdot,r)$.
\end{lemma}

\begin{proof}
\label{6k3}
Denote by $R_{\beta}^{(2)}(y)$ the left hand side of \eqref{6k2}.
By the fact that $H_{\beta}$ is increasing, and by \eqref{derhk},
we have that 
\begin{equation}
R_{\beta}^{(2)}(y)\;<\;S_{\beta}^{(2)}(y)\;\;
\text{and}\;\;\frac{dR_{\beta}^{(2)}}{dy}(y)\;<\;
\frac{dS_{\beta}^{(2)}}{dy}(y)\;\;\mbox{for all }y\;
\in\;(-\infty,\,g_{\beta}]\;.
\end{equation}
Furthermore, 
\begin{equation*}
\lim_{y\rightarrow-\infty}R_{\beta}^{(2)}(y)\;=\;\infty\;\;
\mbox{and}\;\;R_{\beta}^{(2)}(g_{\beta})\;<\;S_{\beta}^{(2)}(g_{\beta})
\;=\; 3 l_{\beta} \;=\; 2/\beta \;<\; 1\;,
\end{equation*}
the equation $R_{\beta}^{(2)}(y)=1$ has at least one solution. By the
first inequality of \eqref{6k3}, by (2) of Lemma \ref{alem4}, and by
the fact that $S_{\beta}^{(2)}(g_{\beta})<1$, the solution of
$R_{\beta}^{(2)}(y)=1$ should be less than $\tilde{g}_{\beta}$.  By
the second inequality of \eqref{6k3} and by (2) of Lemma \ref{alem4},
$R_{\beta}^{(2)}$ is strictly decreasing on
$(-\infty,\,\tilde{g}_{\beta})$.  Hence, the solution of the equation
$R_{\beta}^{(2)}(y)=1$ is unique.

We now prove that the Hessian of $F_{\beta}(\cdot,r)$ at
$(H_{\beta}(y),\,K_{\beta}(y))$ is positive definite for all
$y<\tilde{g}_{\beta}$. Let
\begin{equation*}
b_{0}(y)\;=\;\frac{1}{\beta
  H_{\beta}(y-3r/2)}-\frac{3}{2}\;,\;\;b_{1}(y)\;
=\;\frac{1}{\beta H_{\beta}(y)}-\frac{3}{2}\;,\;\;b_{2}(y)
\;=\;\frac{1}{\beta K_{\beta}(y)}-\frac{3}{2}\;.
\end{equation*}
The positiveness of $\det\big[ (\nabla^{2}F_{\beta})
(H_{\beta}(y),\,K_{\beta}(y),\,r)\big]$ can be proven as in the proof
of Lemma \ref{alem6}.  The trace of the Hessian, which can be written
as $2b_{0}(y)+b_{1}(y)+b_{2}(y)$, is positive since $b_{0}(y)>0$ and
$b_{1}(y)+b_{2}(y)>0$, where the latter inequality follows from
\eqref{HK inv}.
\end{proof}

For $r\in(0,\,r_{1}^{\beta})$, by Lemmata \ref{alem6} and \ref{alem5},
we obtain four additional critical points on
$\{\boldsymbol{x}:x_{1}\neq x_{2}\}$
\begin{align*}
&\boldsymbol{\sigma}_{1}^{\beta}(r)\;=\;
\big(K_{\beta}(y_{1}^{\beta}(r)),\,H_{\beta}(y_{1}^{\beta}(r))\big)
\;,\;\;\boldsymbol{\sigma}_{2}^{\beta}(r)\;=\;
\big(H_{\beta}(y_{1}^{\beta}(r)),\,K_{\beta}(y_{1}^{\beta}(r))\big)\;,\;\;\\
&\quad\boldsymbol{m}_{1}^{\beta}(r)\;=\;
\big(H_{\beta}(y_{2}^{\beta}(r)),\,K_{\beta}(y_{2}^{\beta}(r))\big)
\;,\;\;\boldsymbol{m}_{2}^{\beta}(r)\;=\;
\big(K_{\beta}(y_{2}^{\beta}(r)),\,H_{\beta}(y_{2}^{\beta}(r))\big)\;.
\end{align*}
One can easily verify that this notation coincides with the one
adopted in Section 5.1 for small $r$. On the other hand, for
$r\in(r_{1}^{\beta},\,\infty)$, there are only two critical points
$\boldsymbol{m}_{1}^{\beta}(r),\,\boldsymbol{m}_{2}^{\beta}(r)$ on
$\{\boldsymbol{x}:x_{1}\neq x_{2}\}$ which are local minima of
$F_{\beta}(\cdot,r)$.

\smallskip
\noindent{\bf C. The structure of the valleys.}
We first compare the heights of the local minima of $F_\beta(\cdot,\,r)$. For $r>r_2^\beta$, there are two local minima $\boldsymbol{m}_{1}^{\beta}(r)$ and $\boldsymbol{m}_{1}^{\beta}(r)$, and it is obvious from the symmetry that 
\begin{equation*}
F_{\beta}(\boldsymbol{m}_{1}^{\beta}(r),r)\;=\;
F_{\beta}(\boldsymbol{m}_{2}^{\beta}(r),r);.
\end{equation*}
Hence, it suffices to focus only on the case $r<r_2^\beta$. 

\begin{lemma}
\label{alem7}
For $r\in(0,\,r_{2}^{\beta})$, the three local minima of
$F_{\beta}(\cdot,r)$ satisfy
\begin{equation*}
F_{\beta}(\boldsymbol{m}_{1}^{\beta}(r),r)\;=\;
F_{\beta}(\boldsymbol{m}_{2}^{\beta}(r),r)\;<\;
F_{\beta}(\boldsymbol{m}_{0}^{\beta}(r),r)\;.
\end{equation*}
In particular, $\boldsymbol{m}_{1}^{\beta}(r)$ and
$\boldsymbol{m}_{2}^{\beta}(r)$ are the global minima of
$F_{\beta}(\cdot,r)$ for all $r>0$. 
\end{lemma}

\begin{proof}
Since $\boldsymbol{m}_{i}^{\beta}(0)=\boldsymbol{m}_{i}^{\beta}$,
the three values $F_{\beta}(\boldsymbol{m}_{i}^{\beta}(0),0)$, $0\le
i\le2$, are the same. Hence, it suffices to prove that
\begin{equation}
\label{hm1}
\frac{d}{dr}F_{\beta}(\boldsymbol{m}_{0}^{\beta}(r),r)
\;>\;0\;\;\mbox{and}\;\;\frac{d}{dr}F_{\beta}(\boldsymbol{m}_{1}^{\beta}(r),r)
\;=\;\frac{d}{dr}F_{\beta}(\boldsymbol{m}_{2}^{\beta}(r),r)\;<\;0\;.
\end{equation}
By the chain rule, and by the fact that $\boldsymbol{m}_{i}^{\beta}(r)$,
$0\le i\le2$, are critical points of $F_{\beta}(\cdot,r)$, it is
easy to check that 
\begin{gather*}
\frac{d}{dr}F_{\beta}(\boldsymbol{m}_{0}^{\beta}(r),r)
\;=\;1-3p_{\beta}(r)\\
\frac{d}{dr}F_{\beta}(\boldsymbol{m}_{i}^{\beta}(r),r)
\;=\;H_{\beta}\big(y_{1}^{\beta}(r)-3r/2\big) \;-\;
\frac{H_{\beta}(y_{1}^{\beta}(r))+K_{\beta}(y_{1}^{\beta}(r))}{2}
\;,\;i=1,\,2\;.
\end{gather*}
Assertion \eqref{hm1} follows from the fact that $p_{\beta}(r)<1/3$
and $H_{\beta}\big(y_{1}^{\beta}(r)-3r/2\big)<
H_{\beta}(y_{1}^{\beta}(r))<K_{\beta}(y_{1}^{\beta}(r))$.
\end{proof}

To compare the heights of saddle points, let 
\begin{equation*}
h_{0}^{\beta}(r)\;=\;F_{\beta}(\boldsymbol{\sigma}_{0}^{\beta}(r),r)
\;\;\mbox{and}\;\;h_{1}^{\beta}(r)\;=\;
\begin{cases}
F_{\beta}(\boldsymbol{\sigma}_{1}^{\beta}(r),r) 
& \mbox{if }r\in(0,\,r_{1}^{\beta})\;,\\
F_{\beta}(\boldsymbol{p}^{\beta}(r),r) 
& \mbox{if }r\in(r_{1}^{\beta},\,r_{2}^{\beta})\;.
\end{cases}
\end{equation*}

\begin{lemma}
\label{alem8}
For $r\in(0,\,r_{1}^{\beta})\cup(r_{1}^{\beta},\,r_{2}^{\beta})$,
we have that $h_{0}^{\beta}(r)<h_{1}^{\beta}(r)$. 
\end{lemma}

\begin{proof}
An argument, analogous to the one presented in the proof of Lemma
\ref{alem7}, proves the assertion of the lemma for
$r\in(0,\,r_{1}^{\beta})$. For $r\in(r_{1}^{\beta},\,r_{2}^{\beta})$,
one can check that the function $t\mapsto F_{\beta}(t,t,r)$ is
decreasing on $(u_{\beta}(r),q_{\beta}(r))$. The assertion of the
lemma follows automatically.
\end{proof}

To complete the description of the metastable behavior, we investigate
the structure of the valleys for each value of $r$. We start with an
elementary observation.

\begin{lemma}
\label{alem9}
For $r\in(0,\,r_{1}^{\beta})\cup(r_{1}^{\beta},\,r_{2}^{\beta})$,
define the line $\boldsymbol{l}_{\beta}(r)$ as 
\begin{equation*}
\boldsymbol{l}_{\beta}(r)\;=\;
\begin{cases}
\{\boldsymbol{x}:x_{1}+x_{2}=H_{\beta}(y_{1}^{\beta}(r))
+K_{\beta}(y_{1}^{\beta}(r))\}\cap\Xi 
& \mbox{if }r\in(0,\,r_{1}^{\beta})\\
\{\boldsymbol{x}:x_{1}+x_{2}=2u_{\beta}(r)\}\cap\Xi 
& \mbox{if }r\in(r_{1}^{\beta},\,r_{2}^{\beta})
\end{cases}
\end{equation*}
Then, $F_{\beta}(\cdot,r)$ restricted to $\boldsymbol{l}_{\beta}(r)$
achieves its minimum only at $\boldsymbol{\sigma}_{1}^{\beta}(r)$ and
$\boldsymbol{\sigma}_{2}^{\beta}(r)$ if $r\in(0,\,r_{1}^{\beta})$, and
at $\boldsymbol{p}^{\beta}(r)$ if
$r\in(r_{1}^{\beta},\,r_{2}^{\beta})$.
\end{lemma}

\begin{proof}
For the first part, it can be shown by a simple differentiation that the function 
\begin{equation*}
t\;\mapsto\; F_{\beta}(t,\,H_{\beta}(y_{1}^{\beta}(r))+
K_{\beta}(y_{1}^{\beta}(r))-t,\,r)
\end{equation*}
achieves a minimum only at $t=H_{\beta}(y_{1}^{\beta}(r))$ and
$t=K_{\beta}(y_{1}^{\beta}(r))$. The proof of the second assertion is
similar.
\end{proof}

Next lemma describes the structure of the valleys at heights
$h_{0}^{\beta}(r)$ and $h_{1}^{\beta}(r)$.

\begin{lemma}
\label{str1}
For all $r>0$, the set $\{ \boldsymbol{x}\in\Xi:
F_{\beta}(\boldsymbol{x},r) < h_{0}^{\beta}(r)\} $ has two connected
components, denoted by $W_{\beta}(1,r)$ and $W_{\beta}(2,r)$, such
that $\boldsymbol{m}_{i}^{\beta}(r)\in W_{\beta}(i,r)$, $i=1,\,2$, and
$\overline{W_{\beta}(1,r)}\cap\overline{W_{\beta}(2,r)}=
\{\boldsymbol{\sigma}_{0}^{\beta}(r)\}$.  For small enough $r$, there
is an additional component, represented by $W_{\beta}(0,r)$, containing
$\boldsymbol{m}_{0}^{\beta}$ and satisfying
$\overline{W_{\beta}(0,r)}\cap\overline{W_{\beta}(i,r)}=\varnothing$
for $i=1,\,2$.
\end{lemma}

\begin{proof}
It is obvious that the set $\{\boldsymbol{x}\in\Xi :F_{\beta}
(\boldsymbol{x},r)< h_{0}^{\beta}(r)\}$ is composed of two connected
components, denoted by $W_{\beta}(1,r)$ and $W_{\beta}(2,r)$, such
that $\boldsymbol{m}_{i}^{\beta}(r) \in W_{\beta}(i,r)$, $i=1,\,2$.

By combining the fact that
$F_{\beta}(\boldsymbol{m}_{0}^{\beta}(0),\,0)
<F_{\beta}(\boldsymbol{\sigma}_{0}^{\beta}(0),\,0)$,
$F_{\beta}(\boldsymbol{m}_{0}^{\beta}(r_{2}^{\beta}),
\,r_{2}^{\beta})>F_{\beta}(\boldsymbol{\sigma}_{0}^{\beta}
(r_{2}^{\beta}),\,r_{2}^{\beta})$ and that the map $r\mapsto
F_{\beta}(\boldsymbol{m}_{0}^{\beta}(r),\,r) -
F_{\beta}(\boldsymbol{\sigma}_{0}^{\beta}(r), \,r)$ is increasing, as
shown in the proof of Lemmata \ref{alem7}, \ref{alem8}, we can assert
that there exists $r_{*}^{\beta}<r_{2}^{\beta}$ such that
\begin{equation*}
\begin{cases}
F_{\beta}(\boldsymbol{m}_{0}^{\beta}(r),\,r) \;<\;
F_{\beta}(\boldsymbol{\sigma}_{0}^{\beta}(r),\,r) & \mbox{if }r<r_{*}^{\beta}\;,\\
F_{\beta}(\boldsymbol{m}_{0}^{\beta}(r),\,r) \;\ge\; 
F_{\beta}(\boldsymbol{\sigma}_{0}^{\beta}(r),\,r) & \mbox{if }r\ge r_{*}^{\beta}\;.
\end{cases}
\end{equation*}
Hence, for $r<r_{*}^{\beta}$, we have an additional connected
component, denoted by $W_{\beta}(0,r)$, containing
$\boldsymbol{m}_{0}^{\beta}(r)$, while this component disappears for
$r\ge r_{*}^{\beta}$. Since, for $r<r_{2}^{\beta}$,
$\boldsymbol{m}_{0}^{\beta}(r)$ and
$\{\boldsymbol{m}_{1}^{\beta}(r),\,\boldsymbol{m}_{2}^{\beta}(r)\}$
are on the different sides of the line $\boldsymbol{l}_{\beta}(r)$,
and since, by Lemmata \ref{alem8}, \ref{alem9},
$F_{\beta}(\cdot,r)>h_{0}^{\beta}(r)$ on the line
$\boldsymbol{l}_{\beta}(r)$, the component $W_{\beta}(0,r)$ is
isolated from the other components. Therefore, it is enough to show
that $\overline{W_{\beta}(1,r)}\cap\overline{W_{\beta}(2,r)} =
\{\boldsymbol{\sigma}_{0}^{\beta}(r)\}$ for all $r>0$. The proof of
this fact is analogous to the one of assertion (2) of Proposition
4.4. The details are left to the reader.
\end{proof}

\begin{lemma}
\label{str2}
For $r\in(0,\,r_{1}^{\beta})\cup(r_{1}^{\beta},\,r_{2}^{\beta})$, the
set $\{ \boldsymbol{x}\in\Xi: F_{\beta}(\boldsymbol{x},r) <
h_{1}^{\beta}(r)\}$ consists of two connected components, denoted by
$V_{\beta}(0,r)$ and $V_{\beta}(1,r)$, such that
$\boldsymbol{m}_{0}^{\beta}(r)\in V_{\beta}(0,r)$ and
$\boldsymbol{m}_{1}^{\beta}(r),\,\boldsymbol{m}_{2}^{\beta}(r)\in
V_{\beta}(1,r)$.  The set
$\overline{V_{\beta}(0,r)}\cap\overline{V_{\beta}(1,r)}$ is equal to
$\{\boldsymbol{\sigma}_{1}^{\beta}(r),\,\boldsymbol{\sigma}_{2}^{\beta}(r)\}$
for $r\in(0,\,r_{1}^{\beta})$, and is equal to
$\{\boldsymbol{p}^{\beta}(r)\}$ for
$r\in(r_{1}^{\beta},\,r_{2}^{\beta})$.
\end{lemma}

\begin{proof}
We first consider the case $r\in(0,\,r_{1}^{\beta})$. By
Lemma \ref{str1}, there exist two paths
$\boldsymbol{\gamma}_{i}^{\beta}(r)$, $i=1,\,2$, connecting
$\boldsymbol{\sigma}_{0}^{\beta}(r)$ to
$\boldsymbol{m}_{i}^{\beta}(r)$ and satisfying
$F_{\beta}(\boldsymbol{x},r) \le
F_{\beta}(\boldsymbol{\sigma}_{0}^{\beta}(r),r)$ for all
$\boldsymbol{x}\in\boldsymbol{\gamma}_{i}^{\beta}(r)$. By
concatenating these two paths $\boldsymbol{\gamma}_{1}^{\beta}(r)$ and $\boldsymbol{\gamma}_{2}^{\beta}(r)$, we can prove that
$\boldsymbol{m}_{1}^{\beta}(r)$ and $\boldsymbol{m}_{2}^{\beta}(r)$
are in the same connected component $V_{\beta}(1,r)$ of the set $\{
\boldsymbol{x}\in\Xi:F_{\beta}(\boldsymbol{x},r)<h_{1}^{\beta}(r)\} $.
Let $V_{\beta}(0,r)$ be the other component containing
$\boldsymbol{m}_{0}^{\beta}(r)$.  It suffices to show that
$\overline{V_{\beta}(0,r)} \cap \overline{V_{\beta}(1,r)}=
\{\boldsymbol{\sigma}_{1}^{\beta}(r),
\,\boldsymbol{\sigma}_{2}^{\beta}(r)\}$.  By assertion (1) of Lemma \ref{alem9},
the first set is a subset of the second one. It remains to prove the other
inclusion.

By an argument, similar to the one presented in the proof of assertion
(2) of Proposition 4.4, we can construct a path
$\boldsymbol{\delta}_{1}^{\beta}(r)$, connecting
$\boldsymbol{\sigma}_{1}^{\beta}(r)$ and
$\boldsymbol{m}_{0}^{\beta}(r)$ and satisfying
$F_{\beta}(\boldsymbol{x},r)\le
F_{\beta}(\boldsymbol{\sigma}_{1}^{\beta}(r),r)$ for all
$\boldsymbol{x}\in\boldsymbol{\delta}_{1}^{\beta}(r)$, and another
path $\boldsymbol{\delta}_{2}^{\beta}(r)$ connecting
$\boldsymbol{\sigma}_{1}^{\beta}(r)$ and one of the points
$\boldsymbol{m}_{1}^{\beta}(r)$ and $\boldsymbol{m}_{2}^{\beta}(r)$
and satisfying $F_{\beta}(\boldsymbol{x},r)\le
F_{\beta}(\boldsymbol{\sigma}_{1}^{\beta}(r),r)$ for all
$\boldsymbol{x}\in\boldsymbol{\delta}_{2}^{\beta}(r)$. This proves
that $\boldsymbol{\sigma}_{1}^{\beta}(r) \in\overline{V_{\beta}(0,r)}
\cap \overline{V_{\beta}(1,r)}$.  By symmetry,
$\boldsymbol{\sigma}_{2}^{\beta}(r)$ is also an element of the same
set and the proof for $r\in(0,\,r_{1}^{\beta})$ is completed.  The
proof in the case $r\in(r_{1}^{\beta},r_{2}^{\beta})$ is similar and
left to the reader.
\end{proof}

The complete description of the metastable behavior of the random walk
$\bs r_N(t)$ in the case $\beta>2$, $\theta_\text{e}=\pi$ and $r>0$, $r\not =
r^\beta_1$, can be obtained from Lemmata \ref{str1} and \ref{str2} and
the results presented in \cite{LS}.

\subsection{General external field with $\theta_\text{e}=2k\pi/3$}
\label{sec53}

By an analogue computation to the one presented in the previous
subsection, we can rigorously analyze the case
$\theta_\text{e}=2k\pi/3$, $k=0,\,1,\,2$, and $\beta>2$. We do not
repeat the argument here and we only state the main result. 

Assume without loss of generality that $k=0$.  There exists a critical
value $r^\beta$, whose closed form is given by
\begin{equation*}
r^\beta\;=\; h \left( \,\frac{1}{4}+\sqrt{\frac{1}{16}-\frac{1}{9\beta}}\,\right)\;,
\end{equation*}
where $h$ is the function defined in Lemma \ref{alem1}, such that 

\begin{enumerate}
\item [(I)] For $r\in(0,\,r^{\beta})$, we observe the phenomenon
  described in case I of Subsection 5.1.

\item [(II)] For $r\in(r^\beta,\,\infty)$, there is only one critical
  point $\bs{m}_0^\beta (r)$, which is the global minimum. This regime
  is illustrated by the left graph of Figure \ref{fig7}.
\end{enumerate}

The regime (II) is clearly different from the high temperature regime
$\beta<\beta_{3}$ with zero external field in which the entropy
prevails. In the present situation, the spins of the configurations
corresponding to the unique global minimum are highly concentrated on
one spin $\bs{v}_0$, while in the high temperature regime with no
external field the spins are equally distributed among the three
possible values.

We conclude this subsection explaining why there is no intermediate
regime. In the case $\theta_\text{e}=0$, for instance, the study of critical
points on the line $\{\bs{x}:x_1=x_2\}$ is related to the solution of
$\tilde{f}_r (t)=\beta$ where
\begin{equation*}
\tilde{f}_r (t)\;=\;\frac {2}{3(1+r-3t)} \log{\frac{1-2t}{t}}\;.
\end{equation*}
Notice that $\tilde{f}_r (t)$ is obtained by flipping the sign in
front of $r$ in the definition of $f_r(t)$. For $r\in(0,\,r^\beta)$,
as in the discussion after Lemma \ref{alem1}, there are three
solutions $\tilde{p}_\beta (r)<\tilde{u}_\beta (r)<\tilde{q}_\beta
(r)$, while there is only one solution $\tilde{p}_\beta (r)$ for
$r\in(r^\beta,\,\infty)$.

In the case $\theta_\text{e}=\pi$, the second critical value
$r_1^\beta$ was obtained as a solution of $1/(\beta u_\beta(r))=3/2$
(cf. \eqref{rcr}), around which the critical point
$(u_\beta(r),\,u_\beta(r))$ is changed from the local maximum to the
saddle point. However, in the case $\theta_\text{e}=0$, this kind of
discontinuity does not appear since $\tilde{u}_\beta(r)>(1+r)/3>1/3$
so that $1/(\beta \tilde{u}_\beta (r))<3/2$ for all
$r\in(0,\,r_1^\beta)$. This explains why there is no intermediate
phase for $\theta_\text{e}=0$.


\begin{thebibliography}{99}

\bibitem{BL1} J. Beltr\'{a}n, C. Landim: Tunneling and metastability
  of continuous time Markov chains. J. Stat. Phys. {\bf 140},
  1065-1114, (2010)

\bibitem{BL2} J. Beltr\'{a}n, C. Landim: Tunneling and metastability
  of continuous time Markov chains II. J. Stat. Phys. {\bf 149},
  598-618, (2012)

\bibitem{BL3} J. Beltr\'{a}n, C. Landim: A Martingale approach to
  metastability. Probab. Theory Related Fields {\bf 161}, 267--307
  (2015)

\bibitem{Ber} N. Berglund: Kramers' law: validity, derivations and
  generalisations.  Markov Processes Relat. Fields {\bf 19}, 459--490
  (2013)

\bibitem{BR} F. Bouchet, J. Reygner: Generalisation of the
  Eyring-Kramers transition rate formula to irreversible diffusion
  processes. preprint (2015) http://arxiv.org/abs/1507.02104

\bibitem{BEGK1} A. Bovier, M. Eckhoff, V. Gayrard, M. Klein:
  Metastability in reversible diffusion process I. Sharp asymptotics
  for capacities and exit times. J. Eur. Math. Soc. {\bf 6}, 399--424
  (2004)

\bibitem{BEGK2} A. Bovier, M. Eckhoff, V. Gayrard, M. Klein:
  Metastability in stochastic dynamics of disordered mean-field
  models.  Probab. Theory Relat. Fields {\bf 119}, 99--161 (2001)

\bibitem{BH} A. Bovier, F. den Hollander: {\sl Metastability: a
    potential-theoretic approach}.  Grundlehren der mathematischen
  Wissenschaften {\bf 351}, Springer, Berlin, 2015.

\bibitem{CGOV} M. Cassandro, A. Galves, E. Olivieri, M. E. Vares:
  Metastable behavior of stochastic dynamics: a pathwise
  approach. J. Stat. Phys. {\bf 35}, 603--634 (1984)


\bibitem{GL} A. Gaudilli\`ere, C. Landim: A Dirichlet principle for
  non reversible Markov chains and some recurrence theorems.
  Probab. Theory Related Fields {\bf 158}, 55--89 (2014)


\bibitem{Lan2} C. Landim: Metastability for a Non-reversible Dynamics:
  The Evolution of the Condensate in Totally Asymmetric Zero Range
  Processes. Commun. Math. Phys. {\bf 330}, 1--32 (2014)

\bibitem{Lan1} C. Landim: A topology for limits of Markov chains.
  Stoch. Proc. Appl. {\bf 125}, 1058--1098 (2014)

\bibitem{LMT} C. Landim, R. Misturini, K. Tsunoda: Metastability of
  reversible random walks in potential field.  J. Stat. Phys. {\bf
    160} 1449--1482 (2015)

\bibitem{LS} C. Landim, I. Seo: Metastability of non-reversible random
  walks in a potential field, the Eyring-Kramers transition rate
  formula. Submitted. arXiv:1605.01009 (2016)

\bibitem{Mis}  R. Misturini: Evolution of the ABC model among the segregated
  configurations in the zero temperature limit. To appear in
  Ann. Inst. H. Poincar\'e, Probab. Stat. arXiv:1403.4981 (2014)

\bibitem{OV} E. Olivieri, M. E. Vares: {\it Large deviations and
    metastability}. In: Encyclopedia of Mathematics and its
  Applications, vol. 100. Cambridge University Press, Cambridge 2005

\bibitem{Pot} R. B. Potts: Mathematical investigation of some
  cooperative phenomena, Ph.D. Thesis, University of Oxford (1950)

\bibitem{Slo} M. Slowik: A note on variational representations of
  capacities for reversible and nonreversible Markov
  chains. unpublished, Technische Universit\"{a}t Berlin, 2012
  
\bibitem{Wu} F. Y. Wu: The Potts model. Rev. Mod. Phys. {\bf 54}, 235--268 (1982)

\end{thebibliography}
\end{document}